\documentclass[12pt]{amsart}
\usepackage{mathtools}
\usepackage{fullpage}
\usepackage[alphabetic,nobysame, initials]{amsrefs}
\usepackage{enumerate}
\usepackage{hyperref}
\usepackage{color}
\usepackage[pdftex]{graphicx}

\newtheorem{Theorem}{Theorem}[section]
\newtheorem{Lemma}[Theorem]{Lemma}

\newtheorem{Corollary}[Theorem]{Corollary}
\newtheorem{Conjecture}{Conjecture}
\newtheorem{Problem}{Problem}
\newtheorem{Definition}{Definition}
\newtheorem{Remark}{Remark}

\renewcommand{\Re}{\mathbb R}
\newcommand{\E}{\mathbb E}

\newcommand{\BB}{\mathbf B}
\newcommand{\CC}{\mathbf C}
\newcommand{\KK}{\mathbf K}

\newcommand{\QQ}{\mathbf Q}
\newcommand{\TT}{\mathbf T}
\newcommand{\M}{\mathbb{M}}
\newcommand{\K}{\mathcal{K}}
\newcommand{\T}{\mathcal{T}}
\newcommand{\C}{\mathcal{C}}

\newcommand{\F}{\mathcal{F}}

\newcommand{\xx}{\mathbf x}
\newcommand{\yy}{\mathbf y}
\newcommand{\pp}{\mathbf p}
\newcommand{\cc}{\mathbf c}
\newcommand{\qqq}{\mathbf q}

\newcommand{\uu}{\mathbf u}

\newcommand{\oo}{\mathbf o}
\newcommand{\Sph}{\mathbb{S}}
\newcommand{\HH}{\mathbb{H}}

\renewcommand{\Re}{{\mathbb R}}

\newcommand{\Ee}{{\mathbb E}}
\newcommand{\Ed}{\Ee^d}

\newcommand{\bd}{{\rm bd}}

\newcommand{\noshow}[1]{}
\newcommand{\Sedm}{{\mathbb S}^{d-1}}
\newcommand{\st}{\; : \; }
\newcommand{\Ze}{{\mathbb Z}}

\newcommand{\conv}{{\rm conv}}

\DeclareMathOperator{\proj}{proj}

\DeclareMathOperator{\area}{area}
\DeclareMathOperator{\perim}{per}
\DeclareMathOperator{\diam}{diam}
\DeclareMathOperator{\mw}{mw}
\newcommand{\G}{\mathcal{G}}
\newcommand{\orn}{\mathbf{o}}

\newcommand{\MM}{\mathbf{M}}
\DeclareMathOperator{\arcsinh}{arcsinh}

\title{On separability in discrete geometry}

\author{K\'{a}roly Bezdek and Zsolt L\'angi}

\address{
K\'{a}roly Bezdek,
Department of Mathematics and Statistics, University of Calgary, Canada,
and
Department of Mathematics, University of Pannonia, Veszpr\'em, Hungary.
}
\email{bezdek@math.ucalgary.ca}

\address{
Zsolt L\'angi,
Department of Algebra and Geometry, Budapest University of Technology and Economics, Hungary,
and
Alfr\'ed R\'enyi Institute of Mathematics, Budapest, Hungary.
}
\email{zlangi@math.bme.hu}


\begin{document}
\begin{abstract}
A problem of Erd\H os (\cite{GoGo45}) and a theorem of G. Fejes T\'oth and L. Fejes T\'oth (\cite{FeFe}) initiated the study of non-separable arrangements of convex bodies and the investigation of totally separable packings of convex bodies with both topics analyzing the concept of separability from the point view of discrete geometry. This article surveys the progress made on these and some closely related problems and highlights the relevant questions that have been left open.
\end{abstract}
\maketitle

\section{Introduction}

Separation by hyperplanes is a basic and often used concept in convex geometry (\cite{Sch14}). 
So, perhaps not surprisingly, separability properties of families of convex sets play an important role 
in discrete geometry as well. In particular, there are a number of recent results 
that have been discovered in response to the questions about non-separable arrangements (resp., totally separable packings) raised by Erd\H os \cite{GoGo45} (resp., G. Fejes T\'oth and L. Fejes T\'oth \cite{FeFe}).
Thus, we felt that it is timely to survey them and orient the readers attention to those questions that have been left open
and seem to have the potential to motivate further research. It is convenient, as well as natural, to group the results into three sections, each representing a topic of independent interest.
Section~\ref{topic1} discusses the state of the art of minimal coverings of non-separable arrangements of convex bodies (resp., spherical caps) in Euclidean spaces (resp., spherical spaces). Section~\ref{topic2} surveys dense totally separable translative packings of convex bodies in Euclidean spaces and densest totally separable packings of congruent spherical caps in the spherical plane. Finally, Section~\ref{topic3} gives an overview on bounds for contact numbers of locally separable sphere packings in Euclidean spaces. 

Throughout, we use some standard notations and concepts from convex geometry (\cite{Sch14}) as well as  discrete geometry (\cite{Bez10}). This includes the following basic ones. The Euclidean norm of a vector (i.e., a point) $\mathbf{p}$ in the $d$-dimensional Euclidean 
space $\Ed$ is denoted by $|\mathbf{p}|:=\sqrt{\langle\mathbf{p}, \mathbf{p}\rangle}$, where $\langle\cdot,\cdot\rangle$ stands for the 
standard inner product. We denote the $(d-1)$-dimensional unit sphere centered at the origin 
$\mathbf{o}\in\Ee^d$ by $\Sedm:=\{\mathbf{u}\in\Ee^d\st |\mathbf{u}|=1\}$. We shall use ${\rm vol}_d(\cdot)$ for the $d$-dimensional volume (i.e., Lebesgue measure) of sets in $\Ee^d$ and use ${\rm area}(\cdot)$  instead of ${\rm vol}_2(\cdot)$ in $\Ee^2$. The closed $d$-dimensional Euclidean ball of radius $\rho>0$ centered at $\mathbf{p}\in\Ed$ is denoted 
by $\BB[\mathbf{p},\rho]:=\{\mathbf{q}\in\Ee^d\st |\mathbf{p}-\mathbf{q}|\leq\rho\}$  with ${\rm vol}_d(\BB[\mathbf{p},\rho])=\rho^d\kappa_d$, where $\kappa_d:={\rm vol}_d\left({\BB[\mathbf{o},1]}\right)$. 
Moreover, ${\rm bd}(\cdot), {\rm int}(\cdot)$, and ${\rm conv}(\cdot )$ refer to the boundary, interior, and the convex hull, respectively, of the corresponding sets. Furthermore,  let $R(\CC)$ (resp., $r(\CC)$) denote the circumradius (resp., inradius) of the compact set $\CC$ with non-empty interior in $\E^d$, which is the radius of the smallest (resp., a largest) $d$-dimensional ball that contains (resp., is contained in) $\CC$.
\section{Minimal coverings of non-separable arrangements}\label{topic1}

\subsection{Non-separable arrangements in Euclidean spaces}
The following theorem was conjectured by Erd\H os and proved by A. W. Goodman and R. E. Goodman in \cite{GoGo45}.

\begin{Theorem}\label{GoGo}
Let the disks $\BB[\xx_1, \tau_1]\subset \Ee^2, \dots , \BB[\xx_n,\tau_n]\subset \Ee^2$ have the following property: No line of $\Ee^2$ divides the disks $\BB[\xx_1, \tau_1], \dots , \BB[\xx_n,\tau_n]$ into two non-empty families without touching or intersecting at least one disk. Then the disks $\BB[\xx_1, \tau_1], \dots , \BB[\xx_n,\tau_n]$ can be covered by a disk of radius $\tau:=\sum_{i=1}^n \tau_i$.
\end{Theorem}

The proof of A. W. Goodman and R. E. Goodman  \cite{GoGo45} is based on the following statement.

\begin{Lemma}\label{lem:GG}
Let $\F = \left\{ [x_i-\tau_i, x_i + \tau_i ] : \tau_i > 0, i=1,2,\ldots, n\right\}$ be a family of closed intervals in $\Re$ such that $\bigcup \F$ is  single closed interval in $\Re$. Let $x := \left( \sum_{i=1}^n \tau_i x_i \right) / \left( \sum_{i=1}^n \tau_i \right)$. Then the interval $\left[ x - \sum_{i=1}^n \tau_i , x + \sum_{i=1}^n \tau_i \right]$ covers $\bigcup \F$.
\end{Lemma}

With the help of Lemma~\ref{lem:GG} it is shown in \cite{GoGo45} that
\[
\bigcup_{i=1}^{n}\BB[\xx_i, \tau_i]\subset\BB\left[ \frac{\sum_{i=1}^n \tau_i \xx_i}{\sum_{i=1}^n \tau_i}, \sum_{i=1}^n \tau_i\right],
\]
finishing the proof of Theorem~\ref{GoGo}.

For a completely different proof of Theorem~\ref{GoGo} see Theorem~6.1 of K. Bezdek and Litvak  \cite{BeLi16}. 

Notice that the value of $\tau$ in Theorem~\ref{GoGo} (as a function of $\tau_1, \dots, \tau_n$) is the smallest possible. Namely, if the centers of the disks $\BB_1, \dots , \BB_n$ lie on a line such that the consecutive disks are tangent to each other, then the radius of the smallest disk containing them is equal to $\sum_{i=1}^n \tau_i$. 

Recall that a {\it convex domain} of $\Ee^2$ is a compact convex set with non-empty interior. Hadwiger \cite{Ha47} extended Theorem~\ref{GoGo} by introducing the concept of {\it non-separable system} as follows: A system of convex domains $\mathbf{K}_i\subset\Ee^2$, $i=1, \dots, n$ is called {\it separable} if there is a line of $\Ee^2$, which is disjoint from each $\mathbf{K}_i$ and which divides $\Ee^2$ into two open half planes each containing at least one $\mathbf{K}_i$. If no such line exists, we call the system {\it non-separable}. Projections on lines and Cauchy's perimeter formula combined with Theorem~\ref{GoGo} yield the following inequalities.

\begin{Theorem}\label{Ha}
Let $\mathbf{K}_i$, $i=1, \dots, n$ be a non-separable system of convex domains in $\Ee^2$. If $\mathbf{K}_0:={\rm conv}\left(\cup_{i=1}^{n}\mathbf{K}_i\right)$ and ${\rm per}\left(\mathbf{K}_i\right), {\rm diam}\left(\mathbf{K}_i\right)$, and $R\left(\mathbf{K}_i\right)$ denote perimeter, diameter, and circumradius, respectively, of the convex domain $\mathbf{K}_i$, $i=0, 1, \dots, n$, then
\begin{equation}\label{Hadwiger}
{\rm per}\left(\mathbf{K}_0\right)\leq\sum_{i=1}^n{\rm per}\left(\mathbf{K}_i\right), {\rm diam}\left(\mathbf{K}_0\right)\leq\sum_{i=1}^n{\rm diam}\left(\mathbf{K}_i\right), \ {\rm and}\ R\left(\mathbf{K}_0\right)\leq\sum_{i=1}^n R\left(\mathbf{K}_i\right).
\end{equation}
\end{Theorem}

It is natural to extend the problem of Erd\H os to higher dimensions as follows.

\begin{Definition}\label{erdos}
Let $\KK$ be a convex body in $\Ee^d$ and let $\K := \{ \xx_i + \tau_i \KK\ |\  \xx_i\in \Re^d, \tau_i>0, i=1,2,\ldots, n\}$, where $d\ge 2$ and $n\ge 2$.
Assume that $\K$ is a {\rm non-separable family}, in short, an {\rm NS-family}, meaning that every hyperplane intersecting $\conv \left( \bigcup \K \right)$ intersects a member of $\K$ in $\Re^d$, i.e., there is no hyperplane disjoint from $\bigcup \K$ that strictly separates some elements of $\K$ from all the other elements of $\K$ in $\Ee^d$. Then, let $\lambda(\K) > 0$ denote the smallest positive value $\lambda$ such that a translate of
$\lambda \left( \sum_{i=1}^n \tau_i \right) \KK$ covers $\bigcup \K$.
\end{Definition}
In terms of Definition~\ref{erdos}, Theorem~\ref{GoGo} states that if $\C$ is an arbitrary NS-family of finitely many homothetic ellipses in $\Ee^2$, then $\lambda(\C)\le 1$. In \cite{BezLan16}, the authors showed that Theorem~\ref{GoGo} holds not only for balls, but also for any centrally symmetric convex body.  In other words, Theorem~\ref{GoGo} extends to balls of finite dimensional normed spaces. (Here we recall that a set $S \subset \Ee^d$ is centrally symmetric with center $\qqq \in \Ee^d$ if $S= 2\qqq-S$, and $\oo$-symmetric if $S=-S$.)

\begin{Theorem}\label{centrally symmetric convex bodies}
For all $d\ge 2$ and $n\ge 2$, and for every $\oo$-symmetric convex body $\KK_0$ and every NS-family $\K := \{ \xx_i + \tau_i \KK_0\ |\  \xx_i\in \Ee^d, \tau_i>0, i=1,2,\ldots, n\}$ the inequality $\lambda(\K) \leq 1$ holds.
\end{Theorem}

The proof of Theorem~\ref{centrally symmetric convex bodies} in \cite{BezLan16} is a short one, which we recall next.

\begin{proof}
Let $\xx = \left( \sum_{i=1}^n \tau_i \xx_i \right) / \left( \sum_{i=1}^n \tau_i \right)$, and set $\KK' = \xx + \left( \sum_{i=1}^n \tau_i \right) \KK_0$. We prove that $\KK'$ covers $\bigcup \K$. For any line $L$ through the origin $\oo$, let $\proj_L : \Re^d \to L$ denote the orthogonal projection onto $L$, and let $h_{\K} : \Sph^{d-1} \to \Re$ and $h_{\KK'} : \Sph^{d-1} \to \Re$ denote the support functions of $\conv\left( \bigcup \K \right)$ and $\KK'$, respectively (for the definition of support function, see \cite{Sch14}).
Then $\proj_L \left( \bigcup \K \right)$ is a single interval, which, by Lemma~\ref{lem:GG}, is covered by $\proj_L \left( \KK' \right)$.
Thus, for any $\uu \in \Sph^{d-1}$, we have that $h_{\K}(\uu) \leq h_{\KK'}(\uu)$, which readily implies that $\bigcup \K \subseteq \KK'$. 
\end{proof}

Perhaps, it is not surprising that A. W. Goodman and R. E. Goodman  \cite{GoGo45} put forward the following conjecture: For every convex body $\KK$ in $\Ee^d$ and every NS-family $\K := \{ \xx_i + \tau_i \KK\ |\  \xx_i\in \Ee^d, \tau_i>0, i=1,2,\ldots, n\}$ the inequality $\lambda(\K) \leq 1$ holds for all $d\ge 2$ and $n\ge 2$. However, the authors of \cite{BezLan16} found the following counterexample to it.

\begin{figure}[h]
\begin{center}
\includegraphics[width=0.3\textwidth]{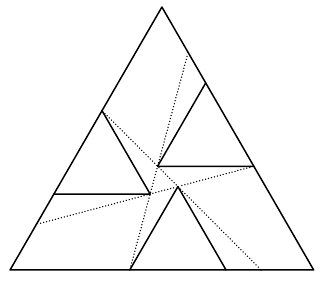}
\caption{A counterexample in the plane for three triangles.}
\label{fig:counterexample2d_3p}
\end{center}
\end{figure}

Place three regular triangles ${\mathcal T}=\{\TT_1, \TT_2, \TT_3\}$ of unit side lengths into a regular triangle $\TT$ of side length $2+2/\sqrt{3} = 3.154700\ldots > 3$ such that
\begin{itemize}
\item each side of $\TT$ contains a side of $\TT_i$, for $i=1,2,3$, respectively (cf. Figure~\ref{fig:counterexample2d_3p}),
\item for $i=1,2,3$, the vertices of $\TT_i$ contained in a side of $\TT$ divide this side into three segments of lengths $2/3 + 1/\sqrt{3}$, $1$, and $1/3 + 1/\sqrt{3}$, in counter-clockwise order.
\end{itemize}
A simple computation yields that the convex hull of any two of the small triangles touches the third triangle, and hence, no two of them can be strictly separated from the third one. Thus, $\lambda({\mathcal T})=2/3 +2/(3\sqrt{3}) = 1.0515\ldots  >1$ and ${\mathcal T}$ is a counterexample to the above conjecture of A. W. Goodman and R. E. Goodman for $n=3$ in $\Ee^2$. Note that for any value $n \geq 4$, placing $n-3$ sufficiently small triangles inside $\TT$ yields a counterexample for $n$ positive homothetic copies of a triangle in $\Ee^2$. As in \cite{BezLan16}, for $d \geq 2$ let $\lambda_{\rm simplex}^{(d)}$ denote the supremum of $\lambda(\K)$, where $\K$ runs over the NS-families of finitely many positive homothetic $d$-simplices in $\Ee^d$. It was proved in \cite{BezLan16} that $\lambda_{\rm simplex}^{(d)}$ is a non-decreasing sequence of $d$. Moreover, Figure~\ref{fig:counterexample3d_3p} shows how to extend the configuration in Figure~\ref{fig:counterexample2d_3p} to $\Ee^3$, implying that $\lambda_{\rm simplex}^{(d)} \geq \lambda_{\rm simplex}^{(2)}\geq 2/3+2/(3\sqrt{3})= 1.0515\ldots>1$ for all $d\ge 3$. The special role of $\lambda_{\rm simplex}^{(d)}$ is highlighted by the following statement proved by the authors in \cite{BezLan16}.

\begin{figure}[h]
\begin{center}
\includegraphics[width=0.3\textwidth]{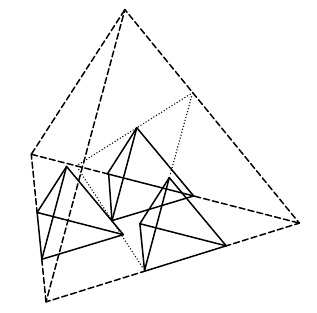}
\caption{A counterexample in $\Re^3$ for three tetrahedra. The dotted lines denote the intersection of a plane touching all three tetrahedra with the smallest tetrahedron containing them.}
\label{fig:counterexample3d_3p}
\end{center}
\end{figure}

\begin{Lemma}\label{simplex bound}
For all $d\ge 2$, $\lambda_{\rm simplex}^{(d)}=\sup_{\K} \lambda(\K)$, where $\K$ runs over the NS-families of finitely many positive homothetic copies of an arbitrary convex body $\KK$ in $\Ee^d$.
\end{Lemma}
As the proof of Lemma~\ref{simplex bound} in \cite{BezLan16} is a short application of Lutwak's containment theorem (\cite{Lu}) and might be of independent interest, we recall it here.

\begin{proof}
As $\lambda_{\rm simplex}^{(d)}\leq\sup_{\K} \lambda(\K)$, it is sufficient to show that for every convex body $\KK$ in $\Ee^d$ and every NS-family $\K = \{ \xx_i + \tau_i \KK\ |\  \xx_i\in \Re^d, \tau_i>0, i=1,2,\ldots, n\}$ a translate of $\lambda_{\rm simplex}^{(d)} ( \sum_{i=1}^n \tau_i ) \KK$ covers $\bigcup \K$. Now, according to Lutwak's containment theorem (\cite{Lu}) if $\KK_1$ and $\KK_2$ are convex bodies in $\Ee^d$ such that every circumscribed simplex of $\KK_2$ has a translate that covers $\KK_1$, then $\KK_2$ has a translate that covers $\KK_1$. (Here a circumscribed simplex of $\KK_2$ means a $d$-simplex of $\Ee^d$ that contains $\KK_2$ such that each facet of the $d$-simplex meets $\KK_2$.) Thus, if $\Delta (\KK)$ is a circumscribed simplex of $\KK$, then $\lambda_{\rm simplex}^{(d)} ( \sum_{i=1}^n \tau_i ) \Delta(\KK)$ is a circumscribed simplex of $\lambda_{\rm simplex}^{(d)} ( \sum_{i=1}^n \tau_i ) \KK$ and $\xx_i + \tau_i \Delta(\KK)$ is a circumscribed simplex of $ \xx_i + \tau_i \KK$ for all $i=1,2,\ldots, n$. Furthermore, $\{ \xx_i + \tau_i \Delta(\KK)\ |\  \xx_i\in \Ee^d, \tau_i>0, i=1,2,\ldots, n\}$ is an NS-family and therefore $\lambda_{\rm simplex}^{(d)} ( \sum_{i=1}^n \tau_i ) \Delta(\KK)$ has a translate that covers $\bigcup \{ \xx_i + \tau_i \Delta(\KK)\ |\  \xx_i\in \Re^d, \tau_i>0, i=1,2,\ldots, n\}\supseteq   \bigcup \K$, which completes the proof.
\end{proof}

In connection with the question whether Theorem~\ref{centrally symmetric convex bodies} can be extended to non-symmetric convex bodies, it was proved in \cite{BezLan16} that if $\K = \{ \xx_i + \tau_i \KK: \xx_i\in \Ee^d, \tau_i>0, i=1,2,\ldots, n\}$ is an arbitrary NS-family of positive homothetic copies of the convex body $\KK$ in $\Ee^d$, then $\lambda(\K) \leq d$ holds for all $n\ge 2$ and $d\ge 2$. This upper bound was recently improved by Akopyan, Balitskiy, and Grigorev \cite{ABG} as follows.

\begin{Theorem}\label{AkBaGr}
Let $\T = \{ \xx_i + \tau_i \TT: \xx_i\in \Ee^d, \tau_i>0, i=1,2,\ldots, n\}$ be an arbitrary family of positive homothetic copies of the $d$-dimensional simplex $\TT$ in $\Ee^d$ such that any hyperplane parallel to a facet of $\TT$ and intersecting ${\rm conv}\left(\bigcup\T\right)$ intersects a member of $\T$. Then it is possible to cover $\bigcup\T$ by a translate of $(d+1)/2 \cdot \left(\sum_{i=1}^n\tau_i\right)\TT$. Moreover, the factor $(d+1)/2$ cannot be improved.
\end{Theorem}

Clearly, Lemma~\ref{simplex bound} and Theorem~\ref{AkBaGr} imply the following corollary.

\begin{Corollary}
If $\K = \{ \xx_i + \tau_i \KK: \xx_i\in \Ee^d, \tau_i>0, i=1,2,\ldots, n\}$ is an arbitrary NS-family of positive homothetic copies of the convex body $\KK$ in $\Ee^d$, then $\lambda(\K) \leq (d+1)/2$ holds for all $n\ge 2$ and $d\ge 2$. 
\end{Corollary}

Hence, the following problem remains open.

\begin{Problem}\label{problem1}
Find $\sup_{\K} \lambda(\K)$ for any given $d\ge 2$, where $\K$ runs over the NS-families of finitely many positive homothetic copies of an arbitrary convex body $\KK$ in $\Ee^d$. In
particular, is there an absolute constant $c>0$ such that $\sup_{\K} \lambda(\K)\leq c$ holds for all $d\ge 2$?
\end{Problem}

\subsection{More on non-separable translative families of convex bodies}

In this subsection we raise a number of isoperimetric-type questions for NS-families of translates of convex bodies, and prove some partial results about them.

\begin{Problem}\label{prob:nonseparable}
Let $\KK$ be a convex body in $\Ed$. Find the maxima of the quantities $\mathrm{vol}_d(\conv\left( \bigcup \F)\right)$, $\mathrm{svol}_{d-1}\left(\conv (\bigcup \F)\right)$, $r\left(\conv (\bigcup \F)\right)$ and $w\left(\conv (\bigcup \F)\right)$, where $\F$ runs over all NS-families of $n>1$ translates of $\KK$, and $\mathrm{svol}_{d-1}(\cdot)$ and $w(\cdot)$ denote the surface volume and the minimal width, respectively, of the corresponding convex body.
\end{Problem}

\begin{Remark}\label{rem:diameter}
It is easy to see that for any NS-family $\F = \{ \xx_i + \lambda_i \KK : \lambda_i > 0, i=1,2, \ldots, n \}$ of positive homothetic copies of a convex body $\KK$ in $\Ed$, we have $\diam \left(\conv (\F) \right)\leq \left( \sum_{i=1}^{n} \lambda_i \right) \diam (\KK)$. Indeed, recall that the diameter of a convex body coincides with its longest orthogonal projection to a line, and hence, the inequality follows from the property that the orthogonal projection of an NS-family of convex bodies to any line is a closed segment. Moreover, we have equality, if and only if, $\KK$ has a diameter $[\pp, \qqq ]$ such that for any $2 \leq i \leq n$, $\xx_{i-1} + \lambda_{i-1} \qqq = \xx_i + \lambda_i \pp$.
\end{Remark}

We start the investigation of Problem~\ref{prob:nonseparable} with the following variant of Kirchberger's theorem (\cite{Sch14}).

\begin{Theorem}\label{thm:nonsepKirchberger}
Let $\F = \{ \KK_i: i=1,2,\ldots, n\}$ be a family of convex bodies in $\Ed$. Let the subfamilies $\F_1$ and $\F_2$ of $\F$ satisfy $\F_1 \cap \F_2 = \emptyset$, $\F_1 \cup \F_2 = \F$. Then there is a hyperplane (strictly) separating $\bigcup \F_1$ from $\bigcup \F_2$ if and only if for any subfamily $\G$ of $\F$ of cardinality at most $d+2$, $\bigcup (\F_1 \cap \G)$ can be (strictly) separated by a hyperplane from $\bigcup (\F_2 \cap \G)$.
\end{Theorem}

\begin{proof}
For any $i=1,2,\ldots, n$ with $\KK_i \in \F_1$, let $\hat{\KK}_i \subset \E^{d+1}$ be the set of points $(\yy,\alpha)$ with $\yy \in \Ed$ and $\alpha \in \Re$ such that $\KK_i$ is contained in the closed half space $\{ \xx \in \Ed : \langle \xx , \yy \rangle \leq \alpha \}$. Then $\hat{\KK}_i$ is convex. Indeed, let $(\yy_1,\alpha_1) , (\yy_2,\alpha_2) \in \hat{\KK}_i$, and $0 \leq t \leq 1$. Consider any $\xx \in \KK$. Then for $s=1,2$, $\langle \xx, \yy_s \rangle \leq \alpha_s$, and  
\[
\langle \xx, t\yy_1 + (1-t) \yy_2 \rangle = t \langle \xx, \yy_1 \rangle + (1-t) \langle \xx, \yy_2 \rangle \leq t \alpha_1 + (1-t) \alpha_2,
\]
implying that $t(\yy_1,\alpha_1) + (1-t)(\yy_2,\alpha_2) \in \hat{\KK}_i$.
It follows similarly that for any $i=1,2,\ldots, n$ with $\KK_i \in \F_2$, the set $\hat{\KK}_i \subset \E^{d+1}$ of points $(\yy,\alpha)$ with $\yy \in \Ed$ and $\alpha \in \Re$ such that $\KK_i$ is contained in the closed half space $\{ \xx \in \Ed : \langle \xx , \yy \rangle \geq \alpha \}$, is convex. Thus, the assertion for nonstrict separation follows from Helly's theorem, and a straightforward modification shows it for strict separation.
\end{proof}

\begin{Corollary}
Let $\F = \{ \KK_i : i=1,2,\ldots, n \}$ be a family of convex bodies in $\Ed$. Then $\F$ is non-separable if and only if the following holds: for any partition of $\F$ into subfamilies $\F_1$ and $\F_2$, there is a subfamily $\G$ of $\F$ of cardinality at most $d+2$ such that $\bigcup (\F_1 \cap \G)$ is not strictly separated from $\bigcup (\F_2 \cap \G)$ by a hyperplane.
\end{Corollary}

Our next theorem answers Problem~\ref{prob:nonseparable} for the case of three Euclidean unit disks.

\begin{Theorem}\label{thm:3disks}
Let $\F=\{ \BB[\xx_i,1] : i=1,2,3 \}$ be an NS-family of three unit disks in $\E^2$. Let $\TT= \conv \{ \xx_1, \xx_2, \xx_3 \}$ and $\CC=\TT+\BB[\orn,1]$. Then
\begin{enumerate}
\item[(a)] $\area(\CC) \leq \pi + 4 + 3\sqrt{3}$, with equality if and only if $\TT$ is a regular triangle whose heights are of length $2$.
\item[(b)] $\perim(\CC) \leq 2\pi + 8$, with equality if an only if $\TT$ is a segment of length $4$.
\item[(c)] $r(\CC) \leq 16/3$, with equality if and only if $\TT$ is a regular triangle whose heights are of length $2$.
\item[(d)] $w(\CC) \leq 4$, with equality if and only if $\TT$ is a regular triangle whose heights are of length $2$.
\end{enumerate}
\end{Theorem}

\begin{proof}
First, we prove (a). By Steiner's formula (see \cite[Chapter 4]{Sch14}), we have
\[
\area(\CC)= \pi + \area(\TT)+\perim(\TT).
\]
First, consider the case that $\TT$ has a non-acute angle $\gamma$. Then the property that $\F$ is non-separable is equivalent to the property that the two sides of $\TT$ enclosing $\gamma$ have lengths at most $2$. Thus, we may assume that these lengths are both equal to $2$. In this case
\[
\area(\CC) = \pi + 4 + 4 \sin \frac{\gamma}{2} + 2 \sin \gamma.
\]
We denote the expression on the right-hand side by $A_0(\gamma)$.
Since the sine function is strictly concave on $(0,\pi)$, it follows that $A_0(\gamma)$ is strictly concave on $[\pi/2,\pi)$, and it has a unique maximum.
Thus, by differentiating, we obtain that $A_0(\gamma)$ is maximal if and only if $\gamma = \frac{2 \pi}{3}$.

Next, consider the case that $\TT$ has no obtuse angles. In this case the property that $\F$ is non-separable is equivalent to the property that no height of $\TT$ has length greater than $2$. For $i=1,2,3$, let the side of $\TT$ opposite of $\xx_i$ be denoted by $S_i$, and its length by $a_i$. Assume that $a_1 \leq a_2 \leq a_3$. Let $h_i$ denote the length of the height of $\TT$ through $\xx_i$. Since $a_i h_i = 2 \area(\TT)$ for all values of $i$, it follows that $h_3 \leq h_2 \leq h_1$, and we may assume that $h_1=2$.
Note that if we move $\xx_1$ on the line parallel to $[\xx_2,\xx_3]$ and away from the bisector of the segment, then $h_1$, as well as $\area(\TT)$, do not change, but $\perim(\TT)$ increases. Thus, if $\area(\CC)$ is maximal, then $h_2=h_1=2$, or $\TT$ has a right angle at $\xx_3$. By the argument in the previous paragraph, it is sufficient to consider the case that $h_1=h_2=2$.

Let $\pi/3 \leq \gamma \leq \pi/2$ denote the angle of $\TT$ at $\xx_3$. Then $a_1 = a_2 = 2/\sin \gamma$, and
\[
\area(\CC) = \pi + \frac{4}{\sin \gamma} + \frac{4 \sin \frac{\gamma}{2}}{\sin \gamma} + \frac{2}{\sin \gamma} = \pi + \frac{6}{\sin \gamma} + \frac{2}{\cos \frac{\gamma}{2}}.
\]
We denote the expression on the right-hand side by $A_1(\gamma)$. By differentiation, we obtain that if $A_1(\gamma)$ is maximal, then $\gamma = \pi/3$ or $\gamma = \pi/2$. Thus, (a) follows from the inequalities $A_1 \left( \pi/3 \right) > A_0\left( 2\pi/3 \right) > A_1 \left( \pi/2 \right)$.

Next, we prove (b).
Observe that
\[
\perim(\CC)= 2\pi +\perim(\TT).
\]
Following the method of the proof of (a), we have that if $\TT$ is not acute, then, by the triangle inequality, $\perim(\TT) \leq 8$. Similarly, if $\TT$ has no obtuse angles, than the maximality of $\perim(\CC)$ yields that the two shortest sides of $T$ have equal lengths, and the length of the corresponding heights is $2$. In this case, denoting by $\gamma$ the angle enclosed by the two legs,
\[
\perim(\TT) = \frac{4}{\sin \gamma} + \frac{4 \sin \frac{\gamma}{2}}{\sin \gamma}.
\]
By differentiating, it follows that the right-hand side is maximal if $\gamma=\frac{\pi}{2}$, and thus, if $\TT$ is not obtuse, then $\perim(\TT) \leq 4 + 2 \sqrt{2} < 8$, implying (b).

To prove (c), we observe that $w(\CC)= w(\TT)+2$ and that $w(\TT)$ coincides with the length of the shortest height of $\TT$.
To prove (d), we note that $r(\CC) = r(\TT)+2$, and apply a case analysis like in the proofs of (a) and (b).
\end{proof}

To present our next result, we need a definition. We note that the same concept for a family of segments was defined by Siegel in \cite{Siegel}.

\begin{Definition}\label{defn:successively}
Let $\F=\{ \KK_i : i=1,2, \ldots, n \}$ be a family of $n$ convex bodies in $\Ed$.
If we can relabel the indices of the elements such that for any $i=2,3,\ldots, n$, $\xx_i + \KK$ intersects $\conv \bigcup_{j=1}^{i-1} (\xx_j + \KK)$, we say that $\F$ is \emph{a successively non-separable family}, in short, an \emph{SNS-family}.
\end{Definition}

\begin{Remark}\label{rem:relations}
Every SNS-family is non-separable, but there are NS-families that are not successively non-separable.
\end{Remark}

\begin{proof}
Let $\F=\{ \KK_i : i=1,2, \ldots, n \}$ be an SNS-family, where the indices are chosen to satisfy the condition in Definition~\ref{defn:successively}. For contradiction, suppose that there is a hyperplane $H$ strictly separating some of the elements of $\F$ from the rest of the elements. Let $i > 1$ denote the smallest index such that $\KK_i$ is strictly separated from $\KK_1$. Then $H$ strictly separates $\KK_i$ from $\bigcup_{j=1}^{i-1} \KK_j$, and thus, also from $\conv \bigcup_{j=1}^{i-1} \KK_j$, contradicting our assumption that $\F$ is an SNS-family. As an example of an NS-family that is not sucessively non-separable, we refer the reader to the configuration in Figure~\ref{fig:counterexample2d_3p}.
\end{proof}

\begin{Theorem}\label{thm:successiveperim}
Let $\F$ be an SNS-family of $n\geq 1$ unit disks in $\E^2$. Then
\[
\perim \left( \conv \bigcup \F \right) \leq 2\pi + 4n-4,
\]
with equality, if and only if, the centerpoints of the elements of $\F$ are all collinear and the distance between any two consecutive centerpoints on this line is $2$.
\end{Theorem}

Note that Theorem~\ref{thm:successiveperim} is trivial for $n=1$. Thus, it readily follows from Lemma~\ref{lem:successiveperim} by applying induction on $n$.

\begin{Lemma}\label{lem:successiveperim}
Let $\KK$ be the convex hull of finitely many unit disks. Let $\BB=\BB[\xx,1]$ satisfy $\BB \cap \KK \neq \emptyset$. Then
\[
\perim(\conv (\KK \cup \BB)) \leq \perim(\KK) + 4,
\]
with equality, if and only if, $\BB$ is tangent to $\KK$, and the centers of all unit disks generating $\KK$, and that of $\BB$, are collinear. 
\end{Lemma}

In the proof we denote the arclength of a curve $\Gamma$ by $\ell(\Gamma)$.

\begin{proof}[Proof of Lemma~\ref{lem:successiveperim}]
Without loss of generality, we may assume that $\BB$ is tangent to $\KK$ at some point $\pp$. Then $\bd (\conv (\KK \cup \BB))$ consists of an arc in $\bd (\KK)$, two segments $[\pp_i, \qqq_i ]$, where $i=1,2$, connecting a point $\pp_i \in \bd(\KK)$ and a point $\qqq_i \in \bd (\BB)$, and a unit circle arc in $\bd (\BB)$ connecting $\qqq_1$ and $\qqq_2$.

Let $\qqq$ denote the point of $\BB$ antipodal to $\pp$, and note that $\qqq \in \bd (\KK \cup \BB)$. For $i=1,2$, let $\Gamma_i$ denote the arc of $\bd (\KK)$ inside $\conv (\KK \cup \BB)$ from $\pp$ to $\pp_i$, and let $\Delta_i$ denote the union of $[\pp_i, \qqq_i]$ and the shorter arc $\Theta_i$ of $\bd(\BB)$ from $\qqq$ to $\qqq_i$. We need to show that for $i=1,2$, $\ell(\Delta_i) \leq \ell(\Gamma_i) + 2$.

In order to do it, observe that $\bd(\KK)$ consists of segments and unit circle arcs. Furthermore, the total length of the unit circle arcs in an arc in $\bd(K)$ is equal to the total turning angle of the arc as we move from one endpoint to the other one. Since at $\pp_i$, $\KK$ and $\conv (\KK \cup \BB)$ have the same tangent line and the tangent line of $\KK$ at $\pp$ and that of $\conv(\KK \cup \BB)$ at $\qqq$ are parallel, the unit circle arcs in $\Gamma_i$ can be obtained by subdividing $\Theta_i$ into finitely many arcs, and translating each part by a vector. Thus, we can translate the segments and unit circle arcs in $\Gamma_i$ to obtain the  union of the translate $\Theta_i'$ of $\Theta_i$ by the vector $\pp - \qqq$, and a polygonal curve connecting $\qqq_i'$ to $\pp_i$, where $\qqq_i'$ is the translate of $\qqq_i$ by $\pp-\qqq$.
Thus, by the triangle inequality, we have
\[
\ell(\Delta_i) = \ell(\Theta_i)+ |\qqq_i-\pp_i| \leq \ell(\Theta_i') + |\pp_i-\qqq_i'| + |\qqq_i'-\qqq_i| \leq \Gamma_i + 2.
\]
Equality here implies that $\qqq_i, \qqq_i', \pp_i$ are collinear, and $\Gamma_i$ is the union of $\Theta_i'$ and $[\pp_i, \qqq_i']$, which clearly yields that $\qqq_i'=\pp_i$ and the length of $\Theta_i$ is $\frac{\pi}{2}$. 
\end{proof}

\begin{Remark}
A straightforward modification of the proof of Lemma~\ref{lem:successiveperim} yields that for any SNS-family $\F$ of $n\geq 1$ translates of an $\orn$-symmetric convex domain $\MM$, we have $\perim_M(\conv(\F)) \leq 4n -4 + \perim_M(\MM)$, where $\perim_M(\cdot)$ denotes perimeter measured in the norm of $\MM$.
\end{Remark}

\begin{Conjecture}
For any NS-family $\F$ of $n\geq 1$ translates of an $\orn$-symmetric convex domain $\MM$, we have $\perim_M(\conv(\F)) \leq 4n -4 + \perim_M(\MM)$.
\end{Conjecture}

Before stating Corollary~\ref{cor:meanwidth}, recall that the mean width $\mw(\KK)$ of the plane convex body $\KK$ satisfies the equality $\perim(\KK) = \pi \mw(\KK)$ (for this relation and the definition of mean width, see e.g. \cite{Gardner}).

\begin{Corollary}\label{cor:meanwidth}
For any SNS-family $\F$ of $n\geq 1$ unit disks,
\[
\mw \left( \conv \bigcup \F \right) \leq 2 + \frac{4n-4}{\pi}, 
\]
with equality, if and only if, the centerpoints of the elements of $\F$ are all collinear and the distance between any two consecutive centerpoints on this line is $2$.
\end{Corollary}

\begin{Corollary}\label{cor:meanwidthanydim}
For any SNS-family $\F$ of $n\geq 1$ unit balls in $\Ed, d\geq 2$, we have
\[
\mw \left( \bigcup \F \right) \leq 2 + \frac{4n-4}{\pi}.
\]
\end{Corollary}

\begin{proof}
By formulas (A.46) and (A.50) of \cite{Gardner}, for any convex body $\KK$ in $\Ed$,
\[
\mw(\KK)= \int_{\G(d,2)} \mw(\KK|S) \, dS,
\]
where $\G(d,k)$ denotes the Grassmannian of the $k$-dimensional linear subspaces of in $\Ed$, $\KK|S$ denotes the orthogonal projection of $\KK$ to the linear subspace $S$, and the integration is with respect to the unique Haar probability measure of $\G(d,2)$. Thus, Corollary~\ref{cor:meanwidthanydim} is an immediate consequence of Theorem~\ref{thm:successiveperim}.
\end{proof}

\subsection{Non-separable arrangements in spherical spaces}
In order to state the main results of this section we need to recall some definitions. A closed cap, in short, a {\it cap}, of spherical radius $\alpha$, for $0\leq \alpha\leq \pi$, is the set of points with spherical distance at most $\alpha$ from a given point in $\Sedm\subset\Ee^d$. A {\it great sphere} of $\Sedm$ is an intersection of $\Sedm$ with a hyperplane of $\Ee^d$ passing through the origin $\oo\in\Ee^d$. Following the terminology of Polyanskii \cite{Po21}, we say that a great sphere {\it avoids} a collection of caps in $\Sedm$ if it does not intersect any cap of the collection. Finally, we say that a finite collection of caps is {\it non-separable} if it does not have a great sphere that avoids the caps such that both hemispheres bounded by it contain at least one cap. Based on these concepts Polyanskii \cite{Po21} proved the following extension of Theorem~\ref{GoGo} to spherical spaces.

\begin{Theorem}\label{Polyanskii}
Let $\mathcal{F}$ be a non-separable family of caps of spherical radii $\alpha_1,\dots ,\alpha_n$ in $\Sedm$, $d\geq 2$. If $\alpha_1+\dots+\alpha_n< \pi/2$, then $\mathcal{F}$ can be covered by a cap of radius $\alpha_1+\dots+\alpha_n$ in $\Sedm$.
\end{Theorem}

It is shown in \cite{Po21} that Theorem~\ref{Polyanskii} is equivalent to Theorem~\ref{Polyanskii-2} below. Polyanskii's proof of Theorem~\ref{Polyanskii-2} uses ideas from \cite{Bal}, \cite{Ball}, and \cite{Bang}. Recall that a closed zone, in short, a {\it zone},  of width $2\alpha$ with $0\leq \alpha\leq \pi/2$ in $\Sedm$ is the set of points with spherical distance at most $\alpha$ from a given great sphere in $\Sedm$. As a {\it spherically convex body} is the intersection of $\Sedm$ with a $d$-dimensional closed convex cone of $\Ee^d$ different from $\Ee^d$, a zone of width $2\alpha$ for $0<\alpha< \pi/2$ is not a spherically convex body.

\begin{Theorem}\label{Polyanskii-2}
Let $\mathbf{Z}_1, \dots ,\mathbf{Z}_n \subset \Sedm$, $d\geq 2$ be zones of width $2\alpha_1,\dots ,2\alpha_n$, respectively, such that $\alpha_1+\dots+\alpha_n<\pi/2$. If $\Sedm\setminus\left(\bigcup_{i=1}^n\mathbf{Z}_i\right)$ possesses at most one pair of antipodal open connected components, then $\bigcup_{i=1}^n\mathbf{Z}_i$ can be covered by a zone of width $2\alpha_1+\dots +2\alpha_n$.
\end{Theorem}

According to \cite{Po21}, Didid suggested to investigate the following more general problem, which he phrased as a conjecture. In what follows, the {\it inradius} of a spherically convex body is the spherical radius of the largest cap contained in it.

\begin{Conjecture}\label{Didid}
Let $\mathbf{Z}_1\subset , \dots ,\mathbf{Z}_n \subset \Sedm$, $d>2$ be zones of width $2\beta_1,\dots ,2\beta_n$, respectively. If $\Sedm\setminus\left(\bigcup_{i=1}^n\mathbf{Z}_i\right)$
consists of $2m$ spherically convex open connected components with inradii $\gamma_1,\dots ,\gamma_{2m}$, respectively, then $2\beta_1+\dots2\beta_n+\gamma_1+\dots+\gamma_{2m} \geq\pi$.
\end{Conjecture}

\begin{Remark}\label{Bezdek-Schneider}
The first named author and Schneider \cite{BezSch} proved that if a cap of spherical radius $\alpha\geq\pi/2$ is covered by a finite family of spherically convex bodies in $\Sedm$, $d>2$, then the sum of the inradii of the spherically convex bodies in the family is at least $\alpha$. This theorem implies Conjecture~\ref{Didid} for the case when $\beta_1=\dots=\beta_n=0$, i.e., $\mathbf{Z}_1 , \dots ,\mathbf{Z}_n$ are great spheres of $\Sedm$. 
\end{Remark}

In connection with Remark~\ref{Bezdek-Schneider}, we call the reader's attention to Problem 7.3.5 of \cite{Be13}, which is still open and so, we include it here.

\begin{Problem}
Prove or disprove that if a cap of spherical radius $0<\alpha<\pi/2$ is covered by a finite family of spherically convex bodies in $\Sedm$, $d> 2$, then the sum of the inradii of the spherically convex bodies in the family is at least $\alpha$.
\end{Problem}

\begin{Remark}
Jiang and Polyanskii \cite{JiPo} proved that if a finite family of zones covers $\Sedm$, then the sum of the widths of the zones in the family is at least $\pi$. Ortega-Moreno \cite{OM} has found another proof of this fact, which was simplified by Zhao \cite{Zh} (see also \cite{GKP}). This theorem implies Conjecture~\ref{Didid} for the case when $\gamma_1=\dots=\gamma_{2m}=0$, i.e., $\Sedm =\bigcup_{i=1}^n\mathbf{Z}_i$.
\end{Remark}

\section{Dense totally separable packings}\label{topic2}

\subsection{Finding the densest totally separable (finite or infinite) translative packings in the Euclidean plane}
The concept of totally separable packings was introduced by G. Fejes T\'oth and L. Fejes T\'oth \cite{FeFe} as follows.

\begin{Definition}\label{defn:totallyseparable}
A packing $\mathcal{F}$ of convex domains in $\Ee^2$ is called a \emph{totally separable packing}, in short, a \emph{TS-packing}, if any two members of $\mathcal{F}$ can be separated by a line which is disjoint from the interiors of all members of $\mathcal{F}$.
\end{Definition}

We shall use the following notion and notation.

\begin{Definition}
Let $\mathcal{F}$ be a packing of convex bodies in $\Ee^d$. The quantity
\[
\limsup_{\rho \to \infty} \frac{{\rm vol}_d(\BB[\oo,\rho] \cap \bigcup \mathcal{F}}{{\rm vol}_d(\BB[\oo,\rho])}
\]
is called the \emph{upper density} of $\mathcal{F}$. If we replace $\limsup$ by $\liminf$ in the above definition, we obtain the \emph{lower density} of $\mathcal{F}$. If the upper and the lower density of $\mathcal{F}$ are equal, we call this common value the \emph{density} of $\mathcal{F}$. 
\end{Definition}

\begin{Definition}
Let $\KK$ be a convex domain in $\Ee^2$. Then let $\square^*(\KK)$ (resp., $\square(\KK)$) denote a minimal area circumscribed quadrilateral (resp., parallelogram) of $\KK$.
\end{Definition} 

G. Fejes T\'oth and L. Fejes T\'oth \cite{FeFe} put forward the problem of finding the largest density of TS-packings by congruent copies of a given convex domain in $\Ee^2$. They have solved this problem for centrally symmetric convex domains as follows.

\begin{Theorem}\label{FeFe-1}
Let $\KK$ be a convex domain in $\Ee^2$ and let $\mathbf{Q}\subset\Ee^2$ be a convex quadrilateral that contains $n>1$ congruent copies of $\KK$ forming a TS-packing in $\mathbf{Q}$.
Then ${\rm area}(\mathbf{Q})\geq n\cdot{\rm area}(\square^*(\KK))$.
\end{Theorem}

According to a theorem of Dowker \cite{Dow}, if $\KK\subset\Ee^2$ is a centrally symmetric convex domain, then among the least area convex quadrilaterals containing $\KK$ there is a parallelogram. Clearly, this observation and Theorem~\ref{FeFe-1} imply the following corollary.

\begin{Corollary}\label{FeFe-2}
Let $\KK$ be a centrally symmetric convex domain and let $\mathcal{P}$ be an arbitrary TS-packing by congruent copies of $\KK$. Then for the (upper) density $\delta(\mathcal{P})$ of $\mathcal{P}$ (i.e., for the fraction $\delta(\mathcal{P})$ of $\Ee^2$ covered by the members of $\mathcal{P}$) we have that

\[
\delta(\mathcal{P})\leq\frac{{\rm area}(\mathbf{K})}{{\rm area}(\square(\KK))}.
\]
Equality is attained for the lattice TS-packing of $\KK$ with fundamental parallelogram $\square(\KK)$.
\end{Corollary}

The line of research started in \cite{FeFe} has been continued by the authors in \cite{BeLa20}. The following close relative of Corollary~\ref{FeFe-2} was proved in \cite{BeLa20}.

\begin{Theorem}\label{BL-1}
If $\delta_{sep}(\KK)$ denotes the largest (upper) density of TS-packings by translates of the convex domain $\KK$ in $\Ee^2$, then
\begin{equation}\label{Bezdek-Langi-1}
\delta_{sep}(\KK) = \frac{{\rm area}(\KK)}{{\rm area}(\square(\KK))}.
\end{equation}
\end{Theorem}

\begin{Remark}
It is worth pointing out that by \eqref{Bezdek-Langi-1} of Theorem~\ref{BL-1}, the densest TS-packing by translates of a convex domain is attained by
a lattice packing. 
\end{Remark}




The following finite TS-packing analogue of Theorem~\ref{BL-1} was also proved in \cite{BeLa20}.

\begin{Theorem}\label{thm:areaformula}
Let $\mathcal{F} = \{ \mathbf{c}_i + \KK : i=1,2,\ldots, n\}$ be a TS-packing by $n$ translates of the convex domain $\KK$ in $\Ee^2$.
Let  $\mathbf{C} = \conv \left(\{ \mathbf{c}_1,\mathbf{c}_2,\ldots, \mathbf{c}_n \}\right)$.
\begin{enumerate}
\item[(\ref{thm:areaformula}.1)]
Then we have
\[
{\rm area} \left(\conv \left( \bigcup_{i=1}^n (\mathbf{c}_i+\KK) \right)\right) = {\rm area} (\mathbf{C}+\KK)\geq \frac{2}{3} (n-1){\rm area} \left(\square(\KK)\right) + {\rm area} (\KK)+\frac{1}{3} {\rm area}(\mathbf{C}).
\]
\item[(\ref{thm:areaformula}.2)]
If $\KK$ or $\mathbf{C}$ is centrally symmetric, then
\[
{\rm area} (\mathbf{C}+\KK)\geq (n-1) {\rm area} \left(\square(\KK)\right) + {\rm area} (\KK).
\]
\end{enumerate}
\end{Theorem}

\begin{Remark}
We note that equality is attained in $(\ref{thm:areaformula}.1)$ of Theorem~\ref{thm:areaformula} for the following TS-packings by translates of a triangle (cf. Figure~\ref{fig:area_triangle}).
Let $\KK$ be a triangle, with the origin $\mathbf{o}$ at a vertex, and $\mathbf{u}$ and $\mathbf{v}$ being the position vectors of the other two vertices, and let $\mathbf{T} = m \KK$, where $m > 1$ is an integer. Let $\mathcal{F}$ be the family consisting of  the elements of the lattice packing $\{ i\mathbf{u}+j\mathbf{v} +\KK : i,j, \in \mathbb{Z} \}$ contained in $\mathbf{T}$.
Then $\mathcal{F}$ is a TS-packing by $n=\frac{m(m+1)}{2}$ translates of $\KK$ with $\conv \left(\bigcup \mathcal{F} \right) = \mathbf{T}=\mathbf{C}+\KK$, where $\mathbf{C}=(m-1)\KK$.
Thus,

\begin{multline*}
{\rm area}(\mathbf{T}) = m^2 {\rm area} (\KK)=\left[\frac{2}{3}m(m+1)-\frac{1}{3}+\frac{1}{3}(m-1)^2\right]{\rm area} (\KK)=\\
=\frac{4}{3}(n-1){\rm area} (\KK)+{\rm area} (\KK)+\frac{1}{3}{\rm area} (\mathbf{C})= \frac{2}{3}(n-1){\rm area} (\square(\KK))+{\rm area}(\KK)+\frac{1}{3}{\rm area}(\mathbf{C}).
\end{multline*}
\end{Remark}

\begin{figure}[ht]
\begin{center}
\includegraphics[width=0.2\textwidth]{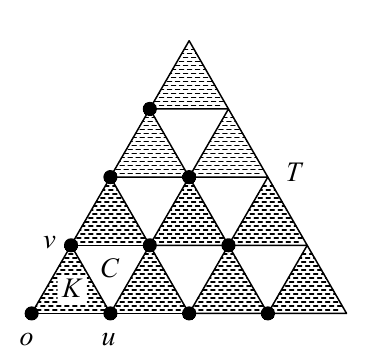}
\caption{An example for equality in (\ref{thm:areaformula}.1).}
\label{fig:area_triangle}
\end{center}
\end{figure}

\begin{Remark}
In (\ref{thm:areaformula}.2) of Theorem~\ref{thm:areaformula} equality can be attained in a variety of ways shown in Figure~\ref{fig:convexhull} for both cases, namely, when $\mathbf{C}$ is centrally symmetric (and $\KK$ is not centrally symmetric, such as a triangle) and when $\KK$ is centrally symmetric (such as a circular disk) without any assumption on the symmetry of $\mathbf{C}$. 
\end{Remark}

\begin{figure}[ht]
\begin{center}
\includegraphics[width=0.6\textwidth]{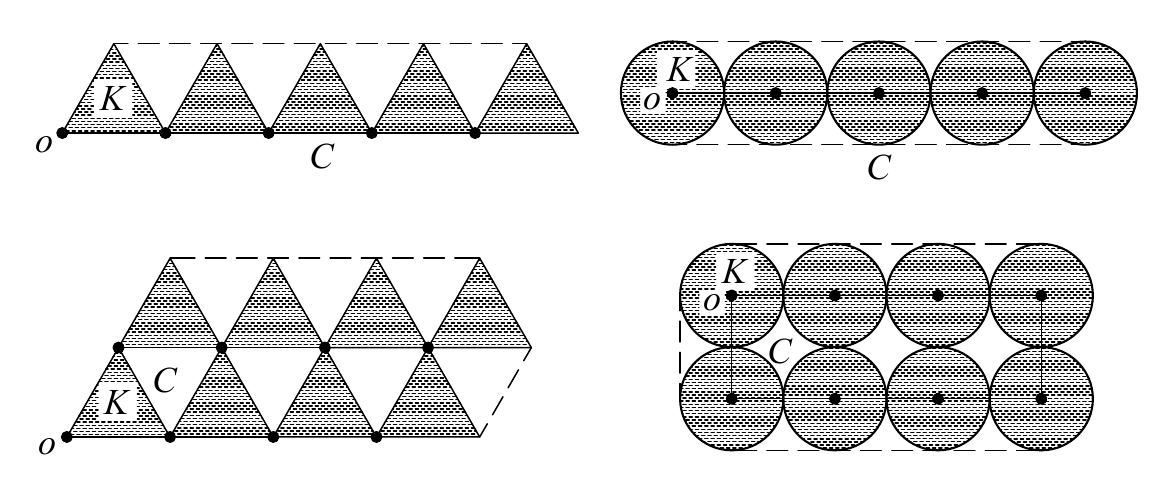}
\caption{TS-packings by translates of a triangle and a unit disk for which equality is attained in (\ref{thm:areaformula}.2) in Theorem~\ref{thm:areaformula}.}
\label{fig:convexhull}
\end{center}
\end{figure}

The following is an easy corollary of (\ref{thm:areaformula}.2) of Theorem~\ref{thm:areaformula}, which is also a close relative of Theorem~\ref{FeFe-1}.

\begin{Corollary}\label{special case}
Let $\mathcal{F} = \{ \mathbf{c}_i + \KK : i=1,2,\ldots, n\}$ be a TS-packing by $n$ translates of the convex domain $\KK$ in $\Ee^2$.
Let  $\mathbf{C} = \conv \left(\{ \mathbf{c}_1,\mathbf{c}_2,\ldots, \mathbf{c}_n \}\right)$. Furthermore, let $\mathbf{Q}\subset\Ee^2$ be a convex quadrilateral that contains $\mathbf{C}+\KK$.
If $\KK$ or $\mathbf{C}$ is centrally symmetric in $\Ee^2$, then ${\rm area}(\mathbf{Q})\geq (n-1) {\rm area} \left(\square(\KK)\right) + {\rm area} \left(\square^*(\KK)\right)$.
\end{Corollary}

The proofs of Theorems~\ref{BL-1} and ~\ref{thm:areaformula} in \cite{BeLa20} are based on a translative TS-packing analogue of Oler's inequality \cite{Oler}. As it might be of independent interest, we describe it as follows. First, we recall the necessary definitions. 

If $\KK$ is an $\oo$-symmetric convex domain in $\Ee^2$, then let $| \cdot |_{\KK}$ denote the {\it norm generated by $\KK$}, i.e., let $|\xx|_{\KK} = \min\{ \lambda : \xx \in \lambda \KK\}$ for any $\xx\in \Ee^2$. The distance between the points $\pp$ and $\qqq$  of $\Ee^2$ measured in the norm $| \cdot |_{\KK}$ is denoted by $|\pp-\qqq |_K$. For the sake of simplicity, the Euclidean distance between the points $\pp$ and $\qqq$ of $\Ee^2$ is denoted by $|\pp-\qqq|$.

If $P = \bigcup_{i=1}^n [\xx_{i-1},\xx_i]$ is a polygonal curve in $\Ee^2$, 
and $\KK$ is an $\oo$-symmetric convex domain in $\Ee^2$, then the {\it Minkowski length} of $P$ is defined as $M_{\KK}(P) = \sum_{i=1}^n |\xx_i-\xx_{i-1}|_{\KK}$. Based on this and using approximation by polygonal curves, one can define the Minkowski length $M_{\KK}(G)$ of any rectifiable curve $G \subseteq \Ee^2$ in the norm $| \cdot |_{\KK}$. 

\begin{Definition}\label{defn:permissiblepolygon}
A closed polygonal curve $P = \bigcup_{i=1}^m [\xx_{i-1},\xx_i]$, where $\xx_0 =\xx_m$, is called \emph{permissible} if there is a sequence
of simple closed polygonal curves $P^n = \bigcup_{i=1}^m [\xx^n_{i-1},\xx^n_i]$, where $\xx^n_0 = \xx^n_m$, satisfying $\xx^n_i \to \xx_i$ for every value of $i$. The interior ${\rm int}(P)$ is defined as $\lim_{n \to \infty} {\rm int} (P^n)$, where the limit is taken according to the topology defined by Hausdorff distance.
\end{Definition}

\begin{Remark}\label{rem:interioriswelldefined}
By the properties of limits, if $P = \bigcup_{i=1}^m [\xx_{i-1},\xx_i]$ is permissible and $P^n$ and $Q^n$ are sequences of simple closed polygonal curves with $\lim_{n \to \infty} P^n = \lim_{n \to \infty} Q^n = P$, then $\lim_{n \to \infty} {\rm int} (P^n) = \lim_{n \to \infty} {\rm int} (Q^n)$, i.e., the interior of a permissible curve is well-defined.
\end{Remark}

One of the main results of \cite{BeLa20} is the following translative TS-packing analogue of Oler's inequality \cite{Oler}.

\begin{Theorem}\label{thm:Oler}
Let $\KK$ be an $\oo$-symmetric convex domain in $\Ee^2$. Let $$\mathcal{F} = \{ \xx_i + \KK : i=1,2,\ldots, n\}$$ be a TS-packing by $n$ translates of $\KK$ in $\Ee^2$, and set $X = \{ \xx_1, \xx_2, \ldots, \xx_n\}$. Furthermore, let $\Pi$ be a permissible closed polygonal curve with the following properties:  \begin{enumerate}
\item the vertices of $\Pi$ are points of $X$

and

\item $X \subseteq \Pi^*$ with $\Pi^* = \Pi \cup {\rm int} (\Pi)$.
\end{enumerate}
Then
\begin{equation}\label{eq:Oler}
\frac{{\rm area} (\Pi^*)}{{\rm area} \left(\square (\KK)\right)} + \frac{M_{\KK}(\Pi)}{4} + 1 \geq n.
\end{equation}
\end{Theorem}

\begin{Remark}\label{equality}
We note that, unlike in Oler's original inequality, equality in (\ref{eq:Oler}) of Theorem~\ref{thm:Oler} is attained in a variety of ways. This is illustrated in Fig.~\ref{fig:Oler}, where the polygon $\Pi$ consists of blocks of zig-zags and simple closed polygons having sides parallel to the two sides of a chosen $\square (K)$. 
\end{Remark}

\begin{figure}[ht]
\begin{center}
\includegraphics[width=0.45\textwidth]{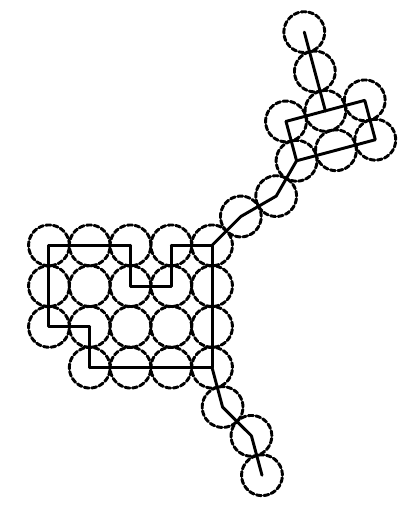}
\caption{A TS-packing of translates of $\KK$ (with $\KK$ being a circular disk for the sake of simplicity), which satisfies the conditions in Theorem~\ref{thm:Oler} and for which there is equality in (\ref{eq:Oler}) of Theorem~\ref{thm:Oler}.}
\label{fig:Oler}
\end{center}
\end{figure}

The following problem raised in \cite{BeLa20} is still open.

\begin{Problem}
Characterize the case of equality in (\ref{eq:Oler}) of Theorem~\ref{thm:Oler}.
\end{Problem}

As an application of Theorem~\ref{thm:Oler}, we outline the proof of (\ref{thm:areaformula}.2) of Theorem~\ref{thm:areaformula} for centrally symmetric $\KK$ following \cite{BeLa20}. It goes as follows.

\begin{proof}[Proof of (\ref{thm:areaformula}.2) of Theorem~\ref{thm:areaformula} for centrally symmetric $K$]
Note that $\bd (\mathbf{C})$ satisfies the conditions in Theorem~\ref{thm:Oler}, and thus, we have
\[
\frac{{\rm area} (\mathbf{C})}{{\rm area} \left(\square (\KK)\right)} + \frac{M_{\KK}\left(\bd (\mathbf{C})\right)}{4} + 1  \geq n.
\]
Next, recall the following inequality from \cite{BeLa20}.
\begin{Lemma}\label{lem:Radon}
Let $\KK$ be a convex domain in $\Ee^2$ and let $\mathbf{Q}$ be a convex polygon. Furthermore, let $A(\mathbf{Q},\KK)$ denote the mixed area of $\mathbf{Q}$ and $\KK$ (for the definition of mixed area, the reader is referred to \cite[Chapter 5]{Sch14}).
If $\KK$ is centrally symmetric, then
\[
\frac{8 A(\mathbf{Q},\KK)}{{\rm area}\left(\square (\KK)\right)} \geq M_{\KK}\left(\bd (\mathbf{Q})\right).
\]
\end{Lemma}

Thus, Lemma~\ref{lem:Radon} yields that
\[
\frac{{\rm area} (\mathbf{C})}{{\rm area} \left(\square (\KK)\right)} + \frac{2 A(\mathbf{C},\KK)}{{\rm area} \left(\square (\KK)\right)} + 1  \geq n.
\]
From this, it follows that
\[
{\rm area} \left(\conv \left( \bigcup_{i=1}^n (\mathbf{c}_i+\KK) \right)\right)= {\rm area} (\mathbf{C}+\KK)={\rm area} (\mathbf{C}) + 2 A(\mathbf{C},\KK) + {\rm area} (\KK)\geq
\]
\[
(n-1) {\rm area} \left(\square (\KK)\right) + {\rm area} (\KK),
\]
finishing the proof.
\end{proof}

\subsection{Finding the densest totally separable packings of a given number of congruent spherical caps in the spherical plane}

The Tammes problem, which is one of the best-known problems of discrete geometry, was originally proposed by the Dutch botanist Tammes  in 1930. It is about finding the arrangement of $N$ points on the unit sphere ${\mathbb S}^2$ such that it maximizes the minimum distance between any two points. Equivalently, for given $N>1$, one is asking for the largest $r$ such that there exists a packing of $N$ spherical caps with angular radius $r$ in the unit sphere ${\mathbb S}^2$. For a short overview of this problem see \cite{BL22}, and for a comprehensive overview we refer the interested reader to the recent article \cite{Mu18}.

The following definition is a natural extension to ${\mathbb S}^2$ of the Euclidean analogue notion of Definition~\ref{defn:totallyseparable}.

\begin{Definition}\label{spherical TS}
A family of spherical caps of ${\mathbb S}^2$ is called a {\rm totally separable packing} in short, a {\rm TS-packing} if any two spherical caps can be separated by a great circle of ${\mathbb S}^2$ which is disjoint from the interior of each spherical cap in the packing. 
\end{Definition}

The analogue of the Tammes problem for TS-packings, called the {\it separable Tammes problem}, was raised in \cite{BL22}.

\begin{Definition}\label{separable-Tammes-1}
For given $k>1$ find the largest $r>0$ such that there exists a TS-packing of $k$ spherical caps with angular radius $r$ in ${\mathbb S}^2$. Let us denote this $r$ by $r_{\rm STam}(k, {\mathbb S}^2)$.
\end{Definition}

\begin{Remark}\label{central-symmetry}
It was noted in \cite{BL22} that $r_{\rm STam}(2k'-1, {\mathbb S}^2)= r_{\rm STam}(2k', {\mathbb S}^2)$ holds for any integer $k'>1$.
\end{Remark}

The following TS-packings are special solutions of the separable Tammes problem (\cite{BL22}). Consider three mutually orthogonal great circles on ${\mathbb S}^2$. These divide the sphere into eight regular spherical triangles of side length $\pi/2$. The family of inscribed spherical caps of these triangles is a TS-packing of eight spherical caps of radius $\arcsin (1/\sqrt{3}) (\approx 35.26^{\circ})$. We call such a family an \emph{octahedral TS-packing} (Figure~\ref{octahedral}). Similarly, the side lines of a regular spherical triangle of side length $\arccos (1/4) \approx 75.52^{\circ}$ divide the sphere into two regular spherical triangles of side length $\arccos (1/4)$, and six isosceles spherical triangles of side lengths $\pi - \arccos (1/4), \pi - \arccos (1/4)$ and $\arccos (1/4)$. 
The inscribed spherical caps of the six isosceles triangles form a TS-packing of six spherical caps of radius $\arctan (3/4) \approx 36.87^{\circ}$, which we call a \emph{cuboctahedral TS-packing} (Figure~\ref{cuboctahedral}).

\begin{figure}[h]
\begin{minipage}[t]{0.48\linewidth}
\centering
\includegraphics[width=0.7\textwidth]{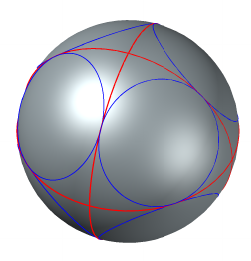}
\caption{An octahedral TS-packing in ${\mathbb S}^2$.}
\label{octahedral}
\end{minipage}
\begin{minipage}[t]{0.48\linewidth}
\centering
\includegraphics[width=0.7\textwidth]{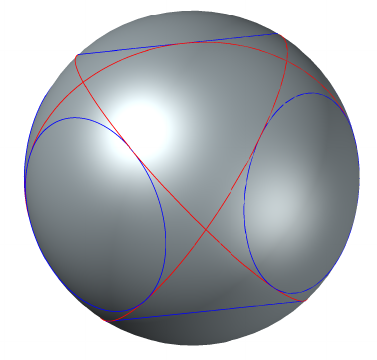}
\caption{A cuboctahedral TS-packing in ${\mathbb S}^2$.}
\label{cuboctahedral}
\end{minipage}
\end{figure}

On the one hand, it is easy to check that  $r_{\rm STam}(2, {\mathbb S}^2)= \pi/2 (=90^{\circ})$ and $r_{\rm STam}(3, {\mathbb S}^2)=r_{\rm STam}(4, {\mathbb S}^2)=\pi/4(=45^{\circ})$. On the other hand, Problem~\ref{separable-Tammes-1} is solved for  $k=5,6,7,8$ in \cite{BL22}.

\begin{Definition}
Let $k >1$ be fixed. A TS-packing of $k$ spherical caps of radius $r_{\rm STam}(k,  {\mathbb S}^2)$ is called \emph{$k$-optimal}.
\end{Definition}

\begin{Theorem}\label{thm:few_caps}
For $5 \leq k \leq 6$ we have $r_{\rm STam}(k, {\mathbb S}^2)= \arctan (3/4)$, and any $k$-optimal TS-packing is a subfamily of a cuboctahedral TS-packing. Furthermore, for $7 \leq k \leq 8$ we have $r_{\rm STam}(k, {\mathbb S}^2)= \arcsin (1/\sqrt{3})$, and any $k$-optimal TS-packing is a subfamily of an octahedral TS-packing.
\end{Theorem}

The value of $r_{\rm STam}(k, {\mathbb S}^2)$ is not known for any $k>8$, which leads to the following problem.

\begin{Problem}
Compute the exact value of $r_{\rm STam}(k, {\mathbb S}^2)$ for some small integers $k>8$.
\end{Problem}

On the other hand, in \cite{BL22} lower and upper bounds were given for $r_{\rm STam}(k, {\mathbb S}^2)$ for any (sufficiently large value of) $k$ as follows.

\begin{Theorem}\label{BL upper bound}
\noindent
\begin{itemize}
\item[(i)] We have
\[
r_{\rm STam}(k, {\mathbb S}^2)\leq\arccos\frac{1}{\sqrt{2}\sin\left(\frac{k}{k-2}\frac{\pi}{4}\right)}
\]

for all $k\geq 5$. In particular,
\[
r_{\rm STam}(8, {\mathbb S}^2)=\arccos\sqrt{\frac{2}{3}} =\arcsin\frac{1}{\sqrt{3}}(\approx 35.26^{\circ}).
\]

\item[(ii)] For any sufficiently large value of $k$, we have $r_{\rm STam}(k, {\mathbb S}^2) \geq \frac{0.793}{\sqrt{k}}$.
\end{itemize}
\end{Theorem}

Notice that the upper bound on $r_{\rm STam}(k, {\mathbb S}^2)$ stated in Part (i) is about $\sqrt{\pi/k}\approx 1.772/\sqrt{k}$ as $k\to+\infty$, which is of the same order of magnitude as the lower bound in Part (ii). Thus, one may wonder whether the following question has a positive answer.

\begin{Problem}
Does the limit $\lim_{k\to+\infty}\sqrt{k}\cdot r_{\rm STam}(k, {\mathbb S}^2)$ exist?
\end{Problem}

\subsection{On densest $\lambda$-separable circle packings in planes of constant curvature}

A natural common generalization of packings and totally separable packings of circular disks in planes of constant curvature was introduced in \cite{AMC} as follows.

\begin{Definition}
Let $0\leq \lambda\leq \rho$. Let $\mathcal{P}$ be a packing of (circular) disks of radius $\rho$ in $\M\in \{ \Ee^2, \HH^2, \Sph^2 \}$, where $ \HH^2$ denotes the hyperbolic plane. We say that $\mathcal{P}$ is a \emph{$\lambda$-separable packing of disks of radius $\rho$} if the family $\mathcal{P'}$ of disks concentric to the disks of $\mathcal{P}$ having radius $\lambda$ form a totally separable packing in $\M$, i.e.,  any two disks of $\mathcal{P'}$ can be separated by a line in $\M$ which is disjoint from the interior of every disk of $\mathcal{P'}$. 
\end{Definition}

Note that a packing of disks of radius $\rho$ in $\M$ is a $0$-separable packing of disks of radius $\rho$ in $\M\in \{ \Ee^2, \HH^2, \Sph^2 \}$. On the other hand, a totally separable packing of disks of radius $\rho$ in $\M$ is a $\rho$-separable packing of disks of radius $\rho$ in $\M\in \{ \Ee^2, \HH^2, \Sph^2 \}$.

It is a natural problem to investigate the maximum density of $\lambda$-separable packings of disks of radius $\rho$ in $\M \in \{ \Ee^2, \HH^2, \Sph^2 \}$.
Before stating the related results from \cite{AMC}, we introduce some notation that we need in their formulation.

Let
\begin{equation}\label{eq:x1s}
x_1^s(y) := \frac{1}{2} \arcsin \frac{\cos \lambda \sin^2 y}{\sqrt{\sin^2 y - \sin^2 \lambda}}, \hbox{ if } 0 < \lambda < \frac{\pi}{4}, \arcsin \tan \lambda < y < \frac{\pi}{2},
\end{equation}
\begin{equation}\label{eq:x2s}
x_2^s(y) := \frac{\pi}{2}-\frac{1}{2} \arcsin \frac{\cos \lambda \sin^2 y}{\sqrt{\sin^2 y - \sin^2 \lambda}}, \hbox{ if } 0 < \lambda < \frac{\pi}{4}, \arcsin \tan \lambda < y < \frac{\pi}{2},
\end{equation}
\begin{equation}\label{eq:xh}
x^h(y) := \frac{1}{2} \arcsinh \frac{\cosh \lambda \sinh^2 y}{\sqrt{\sinh^2 y - \sinh^2 \lambda}}, \hbox{ if } 0 < \lambda < y,
\end{equation}
\begin{equation}\label{eq:xeu}
x^e(y) := \frac{y^2}{2\sqrt{y^2-\lambda^2}}, \hbox{ if } 0 < \lambda < y.
\end{equation}

Furthermore, for any $0 < \lambda < \pi/4, \arcsin \tan \lambda < y < \pi/2$ and $i=1,2$, let $T_i^s(y)$ denote the spherical isosceles triangle of edge lengths $2y, 2x_i^s(y), 2x_i^s(y)$. Similarly, for any $0 < \lambda < y$, let $T^h(y)$ denote the hyperbolic isosceles triangle with edge lengths $2y, 2x^h(y), 2x^h(y)$, and by $T^e(y)$ the Euclidean isosceles triangle with edge lengths $2y, 2x^e(y), 2x^e(y)$. In addition, we denote the regular spherical, hyperbolic, Euclidean triangle of edge length $2\rho$ by $T_{reg}^s(\rho)$,$T_{reg}^h(\rho)$ and $T_{reg}^e(\rho)$, respectively.

\begin{Remark}\label{rem:monotonicityofx}
It is an elementary computation to check that $x_1^s$ as a function of $y$ is strictly decreasing over the closed interval $S_1:= [ \arcsin \tan \lambda, \arcsin (\sqrt{2} \sin \lambda) ]$ and strictly increasing over $S_2:= [\arcsin (\sqrt{2} \sin \lambda), \pi/2]$, with $x_1^s(\arcsin \tan \lambda) = x_1^s (\pi/2) = \pi/4$ and $x_1^s(\arcsin (\sqrt{2} \sin \lambda)) = \lambda$. Since $x_2^s(y)=\frac{\pi}{2}-x_1^s(y)$, similar statements hold for $x_2^s$. For $k=1,2$ and $j=1,2$, we denote the inverse of the restriction of $x_j^s$ to $S_k$ by $x_j^s|_{S_k}^{-1}$. We remark that $\arcsin \tan \lambda < \arcsin (\sqrt{2} \sin \lambda) < \pi/2$ holds for all $\lambda \in (0,\pi/4)$. Furthermore, $x^h(y)$ is a strictly decreasing function of $y$ on $H_1:=(\lambda, \arcsinh (\sqrt{2} \sinh \lambda)]$, and strictly increasing on $H_2:=[\arcsinh (\sqrt{2} \sinh \lambda),+\infty)$, with $x^h(y) \to +\infty$ as $y \to \lambda^+$ or $y \to +\infty$. We denote the inverse of $x^h$ on the two intervals by $x^h|_{H_1}^{-1}$ and $x^h|_{H_2}^{-1}$, respectively.
\end{Remark}

Furthermore, for any $0 < \lambda \leq \arcsin (3/5)$, let 
\begin{equation}\label{eq:yss}
y_s^s(\lambda):=\arcsin \sqrt{\frac{3+5\sin^2 \lambda - \sqrt{9-34 \sin^2 \lambda + 25 \sin^4 \lambda}}{8}} \hbox{ and}
\end{equation}
\begin{equation}\label{eq:ybs}
y_b^s(\lambda):=\arcsin \sqrt{\frac{3+5\sin^2 \lambda + \sqrt{9-34 \sin^2 \lambda + 25 \sin^4 \lambda}}{8}},
\end{equation}
and for any $\lambda > 0$, let
\begin{equation}\label{eq:ysh}
y_s^h(\lambda):=\arcsinh \sqrt{ \frac{5 \sinh^2 \lambda - 3 + \sqrt{25 \sinh^4 \lambda + 34 \sinh^2 \lambda + 9}}{8} }.
\end{equation}

An elementary computation shows that the above expressions exist on the required intervals. 

Finally, in the formulations of Theorems~\ref{thm:densityEu}-\ref{thm:densityH}, for any triangle $T$ in $\M$ with $\M \in \{ \Ee^2, \Sph^2, \HH^2$\}, $\delta(T)$ denotes the area ratio

\[
\frac{\area((\BB_1 \cup \BB_2 \cup \BB_3) \cap T)}{\area(T)},
\]
where $\BB_1, \BB_2$ and $\BB_3$ are disks of radius $\rho$ centered at the vertices of $T$. In the case $\M = \Ee^2$ we also assume that $\rho=1$.

\begin{Definition}
Let $\delta_{\lambda}$ (resp., $\delta_{\lambda}^*$) denote the largest (upper) density of $\lambda$-separable packings of unit disks (resp., $\lambda$-separable lattice packings of unit disks) in $\Ee^2$.
\end{Definition}

It was proved in \cite{FeFe} that $\delta_{1} = \delta_{1}^* =\pi/4$. On the other hand, it is well-known \cite{Lagerungen} that $\delta_{0} = \delta_{0}^* =\pi/\sqrt{12}$. The following theorem extends these results to $\lambda$-separable packings as follows.

\begin{Theorem}\label{thm:densityEu}
We have
\[
\delta_{\lambda} = \delta_{\lambda}^* =
\left\{ \begin{array}{l}
\frac{\pi}{\sqrt{12}}, \hbox{ if } 0 \leq \lambda \leq \frac{\sqrt{3}}{2},\\
\frac{\pi}{4 \lambda}, \hbox{ if } \frac{\sqrt{3}}{2} \leq \lambda \leq 1.\\
\end{array}
\right.
\]
Furthermore, for $0 \leq \lambda \leq \sqrt{3}/2$ and $\sqrt{3}/2 \leq \lambda \leq 1$ the lattice packing $\mathcal{F}$ of unit disks whose Delaunay triangles  are $T_{reg}^e(1)$ and $T^e(\sqrt{2-2\sqrt{1-\lambda^2}})$, respectively, is a $\lambda$-separable packing with density $\delta_{\lambda}$ (see Figure~\ref{fig:density_Eu}). (For the definition of Delaunay decomposition, the reader is referred to the classical book of Gr\"unbaum and Shephard \cite{GrSh1989}.)
\end{Theorem}

\begin{figure}[htb]
  \begin{center}
  \includegraphics[width=0.5\textwidth]{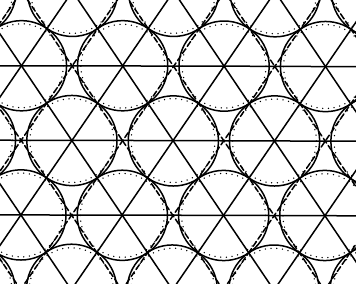} 
 \caption{A densest $\lambda$-separable packing of unit disks in $\Ee^2$ with $\lambda=0.93$. The unit disks and the sides of the Delaunay triangles are denoted by solid lines. The concentric disks of radius $\lambda$ and the lines separating them are drawn with dotted and dashed lines, respectively. The bases of the Delaunay triangles are horizontal, and their length is greater than $2$. The legs of the Delaunay triangles are of length $2$.}
\label{fig:density_Eu}
\end{center}
\end{figure}

The next theorem deals with $\lambda$-separable packings of spherical caps of radius $\rho$ in $ \Sph^2 $. If such a packing contains at least three caps, then $\lambda \leq \pi/4$, and it can be shown that in this case the inequality $\lambda \leq \pi/2-\rho$ is also satisfied.

\begin{Definition}\label{defn:sphericaldensity}
Let $\delta_\lambda^s(\rho)$ denote the largest density of $\lambda$-separable packings of spherical caps of radius $\rho$ in $ \Sph^2 $, i.e., the largest fraction of $ \Sph^2 $ covered by a $\lambda$-separable packing of spherical caps of radius $\rho$ in $ \Sph^2 $.
\end{Definition}

\begin{Theorem}\label{thm:densitySph}
Let $0 \leq \lambda \leq \rho < \pi/2$ with $\rho >0$ and $\lambda \leq \pi/2 - \rho$. Then
\begin{equation}\label{eq:densitylemSph_thm}
\delta_\lambda^s(\rho) \leq
\left\{
\begin{array}{l}
\delta(T_1^s(x_1^s|_{S_1}^{-1}(\rho))), \hbox{ if } \lambda \leq \rho \leq \min \left\{ y_s^s(\lambda), \pi/4 \right\}, \\
\delta(T_{reg}^s(\rho)), \hbox{ if } y_s^s(\lambda) < \rho \leq y_b^s(\lambda),\\
\delta(T_2^s(\rho)), \hbox{ if } \pi/4 < \rho, \hbox{ and } \rho < y_s^s(\lambda) \hbox{ or } \rho > x_b^s(\lambda).
\end{array}
\right.                                                                      
\end{equation}
Furthermore, in all of the above cases we have equality if and only if there is a $\lambda$-separable packing $\F$ of spherical caps of radius $\rho$ whose Delaunay triangles tile $\Sph^2$ and are congruent to $T_1^s(x_1^s|_{S_1}^{-1}(\rho))$, $T_{reg}^s(\rho)$, and $T_2^s(\rho)$, respectively.
\end{Theorem}

Before the next definition we note that it has been observed by B\"or\"oczky \cite{Boroczky} that the notion of density of sphere packings in hyperbolic space
has to be introduced with respect to well-defined underlying cell decompositions of the hyperbolic space. The refined Moln\'ar decomposition of a packing, which is a variant the decomposition defined by Moln\'ar in \cite{Molnar}  for Euclidean circle packings, was introduced for the Euclidean, spherical and hyperbolic plane in \cites{BeLa20, BL22, AMC}, respectively. The next theorem relies on the definition of the hyperbolic variant. Since the definition of this decomposition is fairly involved, instead of defining it here, we refer the reader to the paper \cite{AMC} mentioned above.

\begin{Definition}
Let $\delta_\lambda^h(\rho)$ denote the largest density of $\lambda$-separable packings of hyperbolic disks of radius $\rho$ in the cells of the hyperbolic refined Moln\'ar decompositions in $\HH^2$. 
\end{Definition}
 
Our result for the density of $\lambda$-separable hyperbolic disks is as follows.

\begin{Theorem}\label{thm:densityH}
Let $0 \leq \lambda \leq \rho$, where $0 < \rho$. Then
\begin{equation}\label{eq:densitylemH_thm}
\delta_\lambda^h(\rho) \leq
\left\{
\begin{array}{l}
\delta(T^h(x^h|_{H_1}^{-1}(\rho))), \hbox{ if } \lambda \leq \rho \leq y_s^h(\lambda), \\
\delta(T_{reg}^s(\rho)), \hbox{ if } y_s^h(\lambda) < \rho.\\
\end{array}
\right.                                                                     
\end{equation}
Furthermore, in all of the above cases we have equality if there is a $\lambda$-separable packing of hyperbolic disks of radius $\rho$ whose Delaunay triangles are congruent to $T^h(x^h|_{H_1}^{-1}(\rho))$ or $T_{reg}^h(\rho)$, respectively.
\end{Theorem}

\subsection{On densest separable translative packings in Euclidean spaces of dimensions greater than $\mathbf{2}$}

It is natural to extend the concept of TS-packings, discussed in Definition~\ref{defn:totallyseparable}, to higher dimensions as follows.

\begin{Definition}\label{defn:totallyseparable extended}
A packing $\mathcal{F}$ of convex bodies in $\Ee^d$, $d>2$ is called a \emph{totally separable packing}, in short, a \emph{TS-packing}, if any two members of $\mathcal{F}$ can be separated by a hyperplane of $\Ee^d$ which is disjoint from the interiors of all members of $\mathcal{F}$.
\end{Definition}

An elegant paper of Kert\'esz \cite{Ke88} shows that the (upper) density of any TS-packing of unit diameter balls in $\Ee^3$ is at most $\pi/6$ with equality for the lattice packing of unit diameter balls having integer coordinates. Actually, Kert\'esz \cite{Ke88} proved the following stronger result: If a cube of volume $V>0$ in $\Ee^3$ contains a TS-packing of $N>1$ balls or radius $r>0$, then $V\geq 8 Nr^3$. In fact, it is not hard to see that Kert\'esz's method of proof from \cite{Ke88} implies the following even stronger result.

\begin{Theorem}\label{surface area}
If a cube of volume $V>0$ in $\Ee^3$ is partitioned into $N>1$ convex cells by N-1 successive plane cuts (just one cell being divided by each cut) such that each convex cell contains a ball of radius $r>0$, then the sum of the surface areas of the $N$ convex cells is at least $24Nr^2$ and therefore $V\geq r/3 \cdot \left(24Nr^2\right)=8 Nr^3$.
\end{Theorem}

It would be very interesting to find analogues of Theorem~\ref{surface area} in higher dimensions. 
\begin{Problem}\label{Kertesz-high-dimensions}
Prove or disprove that if a $d$-dimensional cube of volume $V>0$ in $\Ee^d$, $d>3$ is partitioned into $N>1$ convex cells by $N-1$ successive hyperplane cuts (just one cell being divided by each cut) such that each convex cell contains a ball of radius $r>0$, then $V\geq 2^dNr^d$.
\end{Problem}

\begin{Remark}
A positive answer to Problem~\ref{Kertesz-high-dimensions} would imply that the (upper) density of any TS-packing of unit diameter balls in $\Ee^d$, $d>3$ is at most $\kappa_d / 2^d$ with equality for the lattice packing of unit diameter balls having integer coordinates in $\Ee^d$. 
\end{Remark}

Next, recall the following elegant theorem of B\"or\"oczky Jr. \cite{Bo94}: Consider the convex hull $\QQ$ of $n$ non-overlapping translates of an arbitrary convex body $\mathbf{C}$ in $\E^d$ with $n$ being sufficiently large. If $\QQ$ has minimal mean $i$-dimensional projection for given $i$ with $1\leq i<d$, then $\QQ$ is approximately a $d$-dimensional ball (for the definition of mean $i$-dimensional projection, see e.g. \cite{Gardner}). The authors \cite{BeLa19} proved an extension of this theorem to the so-called {\it $\rho$-separable translative packings} of convex bodies in $\E^d$. In short, one can regard $\rho$-separable packings (for $\rho\geq 3$) as packings that are locally totally separable. In what follows, we define the concept of $\rho$-separable translative packings using \cite{BeLa19} and then state the main result of \cite{BeLa19}.


\begin{Definition}\label{local-total-separability}
Let $\mathbf{C}$ be an $\oo$-symmetric convex body of  $\mathbb{E}^d$. Furthermore, let $|\cdot |_{\mathbf{C}}$ denote the norm generated by $\mathbf{C}$.
Now, let $\rho\ge 1$. We say that the packing 
$${\mathcal P}_{{\rm sep}}:=\{\mathbf{c}_i+\mathbf{C}\ |\ i\in I \ {\rm with}\ | \mathbf{c}_j-\mathbf{c}_k |_{\mathbf{C}}\ge 2 \ {\rm for\ all}\ j\neq k\in I\}$$ 
of (finitely or infinitely many) non-overlapping translates of $\mathbf{C}$ with centers $\{\mathbf{c}_i\ |\ i\in I\}$ is a {\rm $\rho$-separable packing} \index{$\rho$-separable packing} in $\mathbb{E}^d$ if for each $i\in I$ the finite packing $\{\mathbf{c}_j+\mathbf{C}\ |\ \mathbf{c}_j+\mathbf{C}\subseteq \mathbf{c}_i+\rho\mathbf{C}\}$ is a TS-packing (in $\mathbf{c}_i+\rho\mathbf{C}$). Finally, let $\delta_{{\rm sep}}(\rho, \mathbf{C})$ denote the largest upper density of all $\rho$-separable translative packings of $\mathbf{C}$ in $\mathbb{E}^d$, \index{$\rho$-separable packing} i.e., let
$$\delta_{{\rm sep}}(\rho, \mathbf{C}):=\sup_{{\mathcal P}_{\rm sep}}\left(\limsup_{\lambda\to+\infty}\frac{\sum_{\mathbf{c}_i+\mathbf{C}\subset\mathbf{W}_{\lambda}^d}{\rm vol}_d(\mathbf{c}_i+\mathbf{C})}{{\rm vol}_d(\mathbf{W}_{\lambda}^d)}\right)\ , $$
where $\mathbf{W}_{\lambda}^d$ denotes the $d$-dimensional cube of edge length $2\lambda$ centered at $\mathbf{o}$ in $\mathbb{E}^d$ with edges parallel to the coordinate axes of $\mathbb{E}^d$.
\end{Definition}

\begin{Remark}\label{density interplay}
Let $\delta(\mathbf{C})$ (resp., $\delta_{{\rm sep}}(\mathbf{C})$) denote the supremum of the upper densities of all translative packings (resp., translative TS-packings) of
the $\oo$-symmetric convex body $\mathbf{C}$ in $\mathbb{E}^d$. Clearly, $\delta_{{\rm sep}}(\mathbf{C})\leq\delta_{{\rm sep}}(\rho, \mathbf{C})\leq \delta(\mathbf{C})$  for all $\rho\geq 1$. Furthermore, if $1\le \rho< 3$, then any $\rho$-separable translative packing of $\mathbf{C}$ \index{$\rho$-separable packing} in $\mathbb{E}^d$ is simply a translative packing of $\mathbf{C}$ and therefore, $\delta_{{\rm sep}}(\rho, \mathbf{C})=\delta(\mathbf{C})$. 
\end{Remark}

Recall that the mean $i$-dimensional projection $M_i(\CC)$ ($i=1,2,\ldots,d-1$) of the convex body $\CC$  in $\E^d$, can be expressed (\cite{Sch14}) with the help of a mixed volume via the formula 
\[
M_i(\CC) = \frac{\kappa_i}{\kappa_d} V(\overbrace{\CC,\ldots,\CC}^i,\overbrace{\BB^d,\ldots,\BB^d}^{d-i}),
\]
where $\kappa_d$ is the volume of $\BB^d$ in $\E^d$. Note that $M_i(\BB^d) = \kappa_i$, and the surface volume of $\CC$ is 
\[
\mathrm{svol}_{d-1}(\CC)= \frac{d \kappa_{d}}{\kappa_{d-1}} M_{d-1}(\CC)
\]
and, in particular, $\mathrm{svol}_{d-1}(\BB^d)=d\kappa_d$. Set $M_d(\CC):={\rm vol}_d(\CC)$. Finally, recall that $R(\CC)$ (resp., $r(\CC)$) denotes the circumradius (resp., inradius) of the convex body $\CC$ in $\E^d$. The following
is the main result of \cite{BeLa19}.

\begin{Theorem}\label{Be-La main-result}
Let $d \geq 2$, $1\leq i\leq d-1$, $\rho\geq 1$, and let $\QQ$ be the convex hull of a $\rho$-separable packing \index{$\rho$-separable packing} of $n$ translates of the $\oo$-symmetric convex body $\CC$ in $\E^d$ such that $M_i(\QQ)$ is minimal and 
\[
n \geq \frac{4^d d^{4d}}{\delta_{{\rm sep}}(\rho,\CC)^{d-1}} \cdot \left( \rho \frac{R(\CC)}{r(\CC)}\right)^d.
\]
Then
\begin{equation}\label{main}
\frac{r(\QQ)}{R(\QQ)} \geq 1 - \frac{\omega}{n^{\frac{2}{d(d+3)}}} \hbox{ for } \omega = \lambda(d) \left( \frac{\rho R(\CC)}{r(\CC)} \right)^{\frac{2}{d+3}}, 
\end{equation}
where $\lambda(d)$ depends only on the dimension $d$. In addition, 
\[
M_i(\QQ) = \left( 1+ \frac{\sigma}{n^{\frac{1}{d}}} \right) M_i(\BB^d) \left( \frac{{\rm vol}_d(\CC)}{\delta_{{\rm sep}}(\rho,\CC) \kappa_d} \right)^{\frac{i}{d}} \cdot n^{\frac{i}{d}},
\]
where 
\[
- \frac{2.25 R(\CC) \rho d i}{r(\CC) \delta_{\rm sep}(\rho,\CC)} \leq \sigma \leq \frac{ 2.1 R(\CC) \rho i}{r(\CC) \delta_{\rm sep}(\rho,\CC)}.
\]
\end{Theorem}

\begin{Remark}
It is worth restating Theorem~\ref{Be-La main-result} as follows: Consider the convex hull $\QQ$ of $n$ non-overlapping translates of an arbitrary $\oo$-symmetric convex body $\CC$ forming a $\rho$-separable packing in $\E^d$ with $n$ being sufficiently large. If $\QQ$ has minimal mean $i$-dimensional projection for given $i$ with $1\leq i<d$, then $\QQ$ is approximately a $d$-dimensional ball.
\end{Remark}

\begin{Problem}
Let $d \geq 2$, $1\leq i\leq d-1$, and let $\CC$ be an $\oo$-symmetric convex body in $\E^d$. Does the analogue of Theorem~\ref{Be-La main-result} hold for translative TS-packings of $\CC$ in $\E^d$?\end{Problem}

\begin{Remark}\label{FTL sausage conjecture}
The nature of the question analogue to Theorem~\ref{Be-La main-result} on minimizing $M_d(\QQ)={\rm vol}_d(\QQ)$ is very different. Namely, recall that Betke and Henk \cite{BeHe98} proved L. Fejes T\'oth's sausage conjecture \index{sausage} for $d\ge 42$ according to which the smallest volume of the convex hull of $n$ non-overlapping unit balls in $\E^d$ is obtained when the $n$ unit balls form a sausage, that is, a linear packing. As linear packings of unit balls are $\rho$-separable\index{$\rho$-separable packing}, therefore the above theorem of Betke and Henk applies to $\rho$-separable packings of unit balls in $\E^d$ for all $\rho\geq 1$ and $d\geq 42$. 
\end{Remark}

We close this section with the following conjecture, which has already been proved for $d=2$ (see (\ref{thm:areaformula}.2) of Theorem~\ref{thm:areaformula}) as well as for all $d\geq 42$ (see Remark~\ref{FTL sausage conjecture}).

\begin{Conjecture}
The volume of the convex hull of an arbitrary TS-packing of $N>1$ unit balls in $\E^d$ with $3\leq d\leq 41$ is at least as large as the volume of the convex hull of $N$ non-overlapping unit balls with their centers lying on a line segment of length 2(N-1).
\end{Conjecture}

\section{On contact numbers for separable translative packings}\label{topic3}

For an overview on contact graphs and contact numbers of packings we refer the interested reader to the recent survey article \cite{BeKh18}. Here we focus on the latest developments for contact numbers of totally (resp., locally) separable translative packings that are not discussed in \cite{BeKh18}. 

Let us recall the definition of contact graphs and contact numbers for packings.

\begin{Definition}\label{contact graph-number}
The {\rm contact graph} $G(\mathcal{P} )$ of a packing $\mathcal{P}$  of convex bodies in $\Ee^d$, $d>1$ is the simple graph whose vertices correspond to the members of the packing, and whose two vertices are connected by an edge if the two members touch each other. The number of edges of $G(\mathcal{P} )$ is called the {\rm contact number} $c(\mathcal{P} )$ of $\mathcal{P}$.
\end{Definition}

The concept of locally separable (sphere) packing was introduced by the first named author in \cite{Bez21}. 

\begin{Definition}
We call the packing $\mathcal{P}$ of convex bodies in $\Ee^d$, $d>1$ a {\rm locally separable packing}, in short, an {\rm LS-packing} if each member of $\mathcal{P}$  together with the members of $\mathcal{P}$ that are tangent to it form a TS-packing.
\end{Definition}

Clearly, any TS-packing is also an LS-packing, but not necessarily the other way around. Moreover, it is worth noting that any $\rho$-separable packing by translates of a convex body for $\rho=3$ is a translative LS-packing and vice versa.

\subsection{On contact numbers for totally separable translative packings}

The concept of {\it separable Hadwiger number} (resp., {\it maximum separable contact number}) was introduced in \cite{BKO} (resp., \cite{BeNa18}) as follows.

\begin{Definition}
Let $\KK$ be a convex body of  $\mathbb{E}^d$, $d>1$. The {\rm separable Hadwiger number} $H_{\rm sep}(\KK)$ of $\KK$ is the maximum number of translates of $\KK$ that all are tangent to $\KK$, and together with $\KK$, form a TS-packing in $\mathbb{E}^d$. Moreover, let $c_{\rm sep}(\KK, n)$ denote the largest contact number of a TS-packing of $n>1$ traslates of $\KK$ in $\mathbb{E}^d$.
\end{Definition}

Clearly, $H_{\rm sep}(\KK)$ is the maximum degree in the contact graph of any totally separable packing of $\KK$. It is proved in \cite{BKO} that $H_{\rm sep}(\KK)=4$ for any smooth convex domain $\KK$ in  $\mathbb{E}^2$. Moreover, the first named author, Khan, and Oliwa \cite{BKO} showed that $c_{\rm sep}(\KK, n)=\lfloor 2n-2\sqrt{n}\rfloor$ for any $n>1$ and any smooth strictly convex domain $\KK$ in $\Ee^2$. In an elegant paper, Nasz\'odi and Swanepoel \cite{NaSw} have completed the characterization of the values of  $H_{\rm sep}(\KK)$ and $c_{\rm sep}(\KK, n)$ for every convex domain $\KK$ in $\mathbb{E}^2$. In order to state their result we need to recall the following notion. The convex domain $\KK\subset\Ee^2$ is called a {\it quasi hexagon} if there exists a parallelogram $\mathbf{P}$ containig $\KK$ such that some two opposite vertices of $\mathbf{P}$ lie on the boundary of $\KK$, and such that each edge of $\mathbf{P}$ has a translate contained in $\KK$. Clearly, triangles, parallelograms, affine regular hexagons, and quarter circular disks are quasi hexagons.

\begin{Theorem}
Let $n>1$ and $\KK$ be a convex domain in $\Ee^2$.
\begin{itemize}
\item[(i)] If $\KK$ is a parallelogram, then $H_{\rm sep}(\KK)=8$ and $c_{\rm sep}(\KK, n)=\lfloor4n-\sqrt{28n-12}\rfloor$.
\item[(ii)] If $\KK$ is a quasi hexagon but not a parallelogram, then $H_{\rm sep}(\KK)=6$ and $c_{\rm sep}(\KK, n)=\lfloor3n-\sqrt{12n-3}\rfloor$.
\item[(iii)] If $\KK$ is not a quasi hexagon, then $H_{\rm sep}(\KK)=4$ and $c_{\rm sep}(\KK, n)=\lfloor2n-2\sqrt{n}\rfloor$.
\end{itemize}
\end{Theorem}

The analogue problem in high dimensions seems to be wide open. We close this section with a higher dimensional partial result proved by the first named author and Nasz\'odi \cite{BeNa18}. For further partial results we refer
the interested reader to \cite{NaSw}.

\begin{Theorem}
Let $n>1$ and $\KK$ be a smooth convex body in $\Ee^d$.
\begin{itemize}
\item[(i)] For $d\in\{2,3,4\}$, we have $H_{\rm sep}(\KK)=2d$ and $c_{\rm sep}(\KK, n)\leq dn-n^{(d-1)/d}f(\KK)$, where $f(\KK)$ depends on $\KK$ only.
\item[(ii)] For all $d\geq 5$, we have $H_{\rm sep}(\KK)\leq 2^{d+1}-3$ and $c_{\rm sep}(\KK, n)\leq (1/2)nH_{\rm sep}(\KK)\leq (2^{d+1}-3)n/2$.

\end{itemize}
\end{Theorem}

\subsection{Largest contact numbers and crystallization for locally separable unit disk packings}
The following result was proved by the first named author in \cite{Bez21}, which is an LS-packing analogue of the crystallization result of Heitmann and Radin \cite{HeRa} and characterizes all LS-packings of $n>1$ unit diameter disks having maximum contact number. Before stating we recall that a \emph{polyomino} is a topological disk obtained by gluing finitely many mutually nonoverlapping unit squares to one another along some of their edges (see e.g. \cite{GrSh1989}).

\begin{Theorem}\label{2D-crystallization}
Let $\mathcal {P}^*$ be an arbitrary LS-packing of $n>1$ unit diameter disks in $\Ee^2$. Then $$c(\mathcal {P}^*)\leq \lfloor 2n-2\sqrt{n}\rfloor,$$
where $\lfloor\cdot\rfloor$ stands for the lower integer part. 
Futhermore, suppose that $\mathcal {P}$ is an LS-packing of $n$ unit diameter disks with $c(\mathcal {P})= \lfloor 2n-2\sqrt{n}\rfloor$, $n\geq 4$ in $\Ee^2$. Let $G_c(\mathcal {P})$ denote the contact graph of $\mathcal {P}$ embedded in $\Ee^2$ such that the vertices are the center points of the unit diameter disks of $\mathcal {P}$ and the edges are line segments of unit length.
Then either $G_c(\mathcal {P})$ is the contact graph of the LS-packing of $7$ unit diameter disks shown in Figure~\ref{Figure1.pdf} or
{\item (i)} $ G_c(\mathcal {P})$ is $2$-connected whose internal faces (i.e., faces different from its external face) form an edge-to-edge connected family of unit squares called a polyomino of an isometric copy of the integer lattice $ \Ze^2$ in $\Ee^2$ (see the first packing in Figure~\ref{Figure1.pdf}) or
{\item (ii)} $G_c(\mathcal {P})$ is $2$-connected whose internal faces are unit squares forming a polyomino of an isometric copy of the integer lattice $ \Ze^2$ in $\Ee^2$ with the exception of one internal face which is a pentagon adjacent along (at least) three consecutive sides to the external face of $ G_c(\mathcal {P})$ and along (at most) two consecutive sides to the polyomino (see the second packing in Figure~\ref{Figure2.pdf}) or 
{\item (iii)}  $G_c(\mathcal {P})$ possesses a degree one vertex on the boundary of its external face such that deleting that vertex together with the edge adjacent to it yields a $2$-connected graph whose internal faces are unit squares forming a polyomino of an isometric copy of the integer lattice $ \Ze^2$ in $\Ee^2$ (see the first packing in Figure~\ref{Figure2.pdf}). 
\end{Theorem}

\begin{figure}[ht]
\begin{center}
\includegraphics[width=0.5\textwidth]{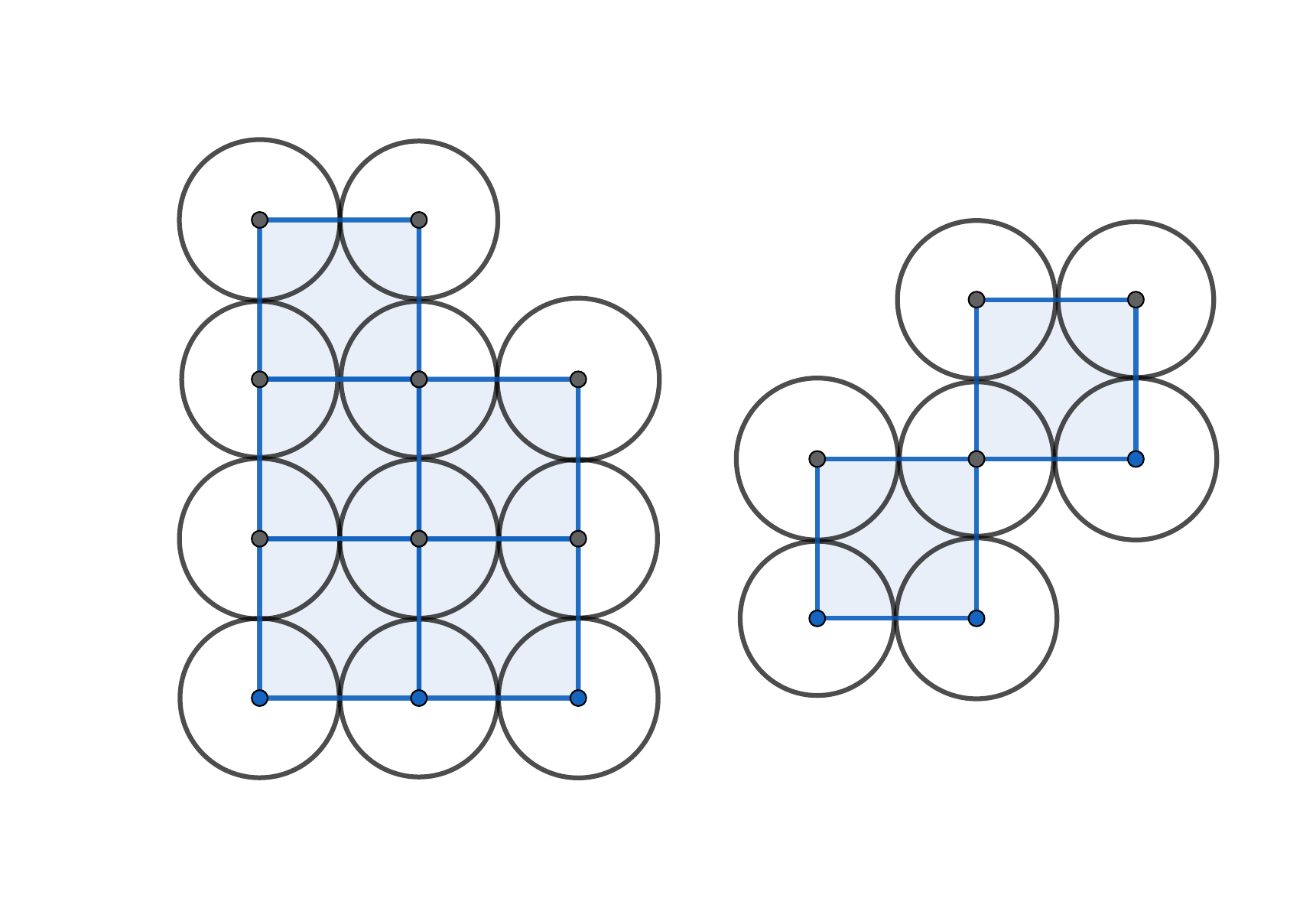}
\caption{An LS-packing of $11$ (resp., $7$) unit diameter disks with maximum contact number.}
\label{Figure1.pdf} 
\end{center}
\end{figure}

\begin{figure}[ht]
\begin{center}
\includegraphics[width=0.5\textwidth]{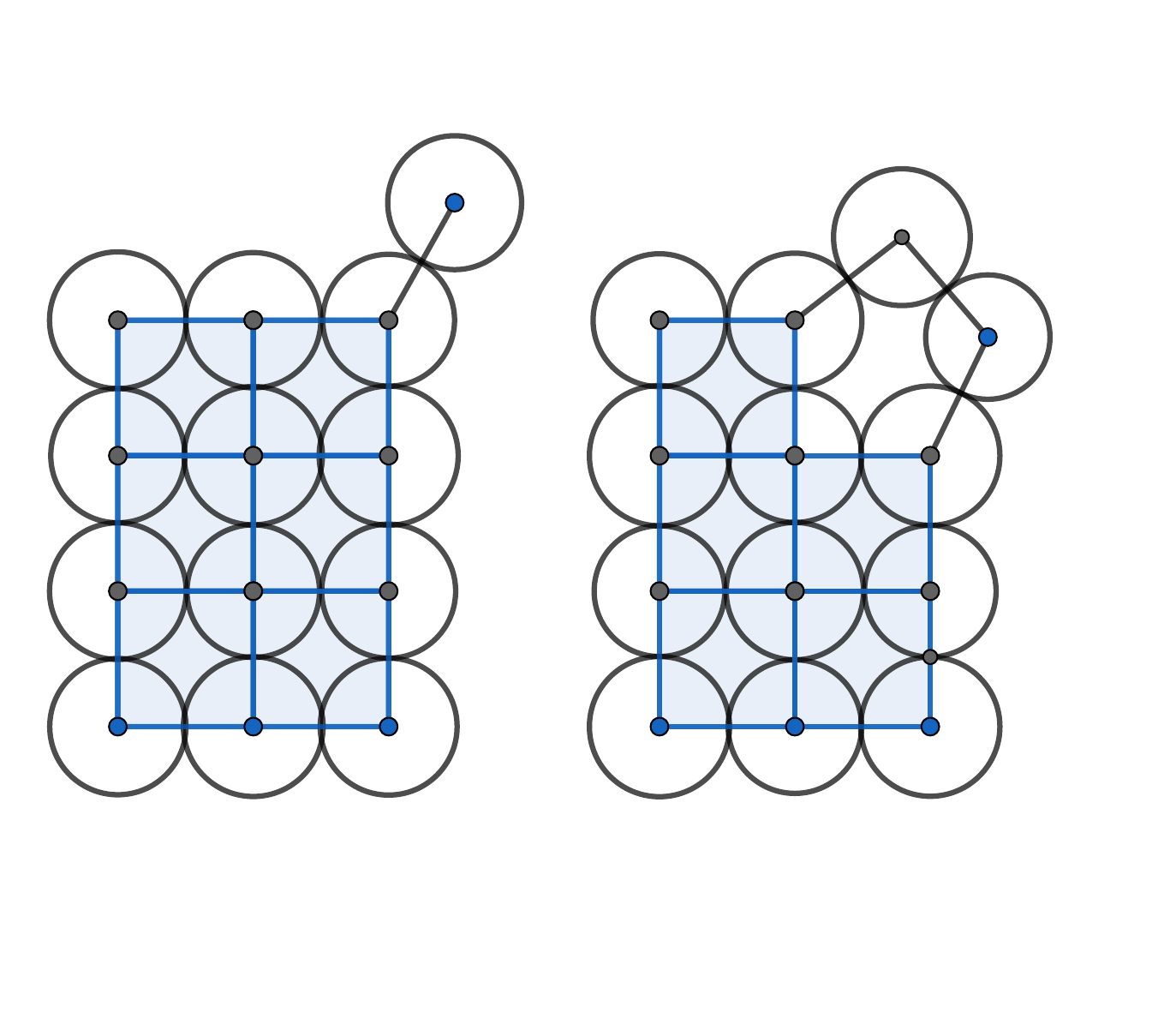}
\caption{Two LS-packings of $13$ unit diameter disks with maximum contact number, the second of which is not a TS-packing.}
\label{Figure2.pdf} 
\end{center}
\end{figure}

The following crystallization result of TS-packings follows from Theorem~\ref{2D-crystallization} in a straightforward way.

\begin{Corollary}\label{2D-corollary-2}
Let $\mathcal {P}^*$ be an arbitrary TS-packing of $n>1$ unit diameter disks in $\Ee^2$. Then $$c(\mathcal {P}^*)\leq \lfloor 2n-2\sqrt{n}\rfloor.$$
Futhermore, suppose that $\mathcal {P}$ is a TS-packing of $n$ unit diameter disks with $c(\mathcal {P})= \lfloor 2n-2\sqrt{n}\rfloor$, $n\geq 4$ in $\Ee^2$. Then either $G_c(\mathcal {P})$ is the contact graph of the LS-packing of $7$ unit diameter disks shown in Figure~\ref{Figure1.pdf}, or
{\item (i)} $ G_c(\mathcal {P})$ is $2$-connected whose internal faces are unit squares forming a polyomino of an isometric copy of the integer lattice $ \Ze^2$ in $\Ee^2$ (see the first packing in Figure~\ref{Figure1.pdf}), or
{\item (ii)}  $G_c(\mathcal {P})$ possesses a degree one vertex on the boundary of its external face such that deleting that vertex together with the edge adjacent to it yields a $2$-connected graph whose internal faces are unit squares forming a polyomino of an isometric copy of the integer lattice $ \Ze^2$ in $\Ee^2$ (see the first packing in Figure~\ref{Figure2.pdf}). 
\end{Corollary}

\subsection{Upper bounding the contact numbers of locally separable unit sphere packings in dimensions larger than $\mathbf{2}$}

In higher dimensions we know less than in the plane. At present the best upper bound for contact numbers of LS-packings of congruent balls in $\Ee^d$, $d\geq 3$ is the one published by the first named author in \cite{Bez21}. We shall need the following notation.
Let $\mathcal {P}:=\{\BB^d[\cc_i,1]| i\in I\}$ be an arbitrary (finite or infinite) packing of unit balls in $\Ee^d$, $d\geq 3$ and let $\mathbf{V}_i:=\{\xx\in \Ee^d | \ |\xx-\cc_i |\leq |\xx-\cc_j |\ {\rm for \ all}\ j\neq i, j\in I\}$ denote the Voronoi cell assigned to $\BB^d[\cc_i,1]$ for $i\in I$. Recall (\cite{Rog}) that the Voronoi cells $\{\mathbf{V}_i | i\in I\}$ form a face-to-face tiling of $\Ee^d$. Then let the largest density of the unit ball $\BB^d[\cc_i,1]$ in its truncated Voronoi cell $\mathbf{V}_i\cap \BB^d[\cc_i,\sqrt{d}]$ be denoted by $\hat{\delta}_d$, i.e., let 
\[
\hat{\delta}_d:=\sup_{\mathcal {P}}\left( \sup_{i\in I}  \frac{\omega_d}{{\rm vol}_d\left(\mathbf{V}_i\cap \BB^d[\cc_i,\sqrt{d}]\right)} \right),
\]
where $\mathcal {P}$ runs through all possible unit ball packings of $\Ee^d$. We are now ready to state the upper bound from \cite{Bez21}.

\begin{Theorem}\label{Be-Main-theorem} Let $\mathcal {P}$ be an arbitrary LS-packing of $n>1$ unit balls in $\Ee^d$, $d\geq 3$. Then
\begin{equation}\label{Main-1}
c(\mathcal {P})\leq \left\lfloor dn-\left(d^{-\frac{d-3}{2}}\hat{\delta}_d^{-\frac{d-1}{d}}\right)n^{\frac{d-1}{d}} \right\rfloor .
\end{equation}
\end{Theorem}

\begin{Remark}\label{Rogers}
Recall the following classical result of Rogers \cite{Ro} (which was rediscovered by Baranovskii \cite{Ba} and extended to
spherical and hyperbolic spaces by B\"or\"oczky \cite{Bo}): Let $\mathcal {P}:=\{\BB^d[\cc_i,1]| i\in I\}$ be an arbitrary packing of unit balls in $\Ee^d$, $d>1$ with $\mathbf{V}_i$ standing for the Voronoi cell assigned to $\BB^d[\cc_i,1]$ for $i\in I$. Furthermore, take a regular $d$-dimensional simplex of edge length $2$ in $\Ee^d$ and then draw a $d$-dimensional unit ball around each vertex of the simplex. Finally, let $\sigma_d$ denote the ratio of the volume of the portion of the simplex covered by balls to the volume of the simplex. Then 
\[
\frac{\omega_d}{{\rm vol}_d\left(\mathbf{V}_i\cap \BB^d[\cc_i,\sqrt{\frac{2d}{d+1}}]\right)}\leq \sigma_d
\] 
holds for all $i\in I$ and therefore $\hat{\delta}_d\leq \sigma_d$ for all $d\geq 3$. The latter inequality and (\ref{Main-1}) yield that if $\mathcal {P}$ is an arbitrary LS-packing of $n>1$ unit balls in $\Ee^d$, $d\geq 3$, then
\[
c(\mathcal {P})\leq \left\lfloor dn-\left(d^{-\frac{d-3}{2}}\hat{\delta}_d^{-\frac{d-1}{d}}\right)n^{\frac{d-1}{d}} \right\rfloor \leq \left\lfloor dn-\left(d^{-\frac{d-3}{2}}{\sigma}_d^{-\frac{d-1}{d}}\right)n^{\frac{d-1}{d}} \right\rfloor,
\]  
where $\sigma_d\sim\frac{d}{e}2^{ -d/2}$ (\cite{Ro}).
\end{Remark}

\begin{Remark}\label{Bezdek}
We note that the density upper bound $\sigma_d$ of Rogers has been improved by the first named author \cite{Bez02} for dimensions $d\geq 8$ as follows: Using the notations of Remark~\ref{Rogers}, \cite{Bez02} showed that
\[
\frac{\omega_d}{{\rm vol}_d\left(\mathbf{V}_i\cap \BB^d[\cc_i,\sqrt{\frac{2d}{d+1}}]\right)}\leq \hat{\sigma}_d
\] 
holds for all $i\in I$ and $d\geq 8$ and therefore $\hat{\delta}_d\leq \hat{\sigma}_d$ for all $d\geq 8$, where $\hat{\sigma}_d$ is a geometrically well-defined quantity satisfying the inequality $\hat{\sigma}_d<\sigma_d$ for all $d\geq 8$. This result and (\ref{Main-1}) yield that if $\mathcal {P}$ is an arbitrary LS-packing of $n>1$ unit balls in $\Ee^d$, $d\geq 8$, then 
\[
c(\mathcal {P})\leq \left\lfloor dn-\left(d^{-\frac{d-3}{2}}\hat{\delta}_d^{-\frac{d-1}{d}}\right)n^{\frac{d-1}{d}} \right\rfloor \leq \left\lfloor dn-\left(d^{-\frac{d-3}{2}}{\hat{\sigma}}_d^{-\frac{d-1}{d}}\right)n^{\frac{d-1}{d}} \right\rfloor .
\] 
\end{Remark}

\begin{Remark}\label{Hales}
The density upper bound $\sigma_3$ of Rogers has been improved by Hales \cite{Ha} as follows:  If $\mathcal {P}:=\{\BB^3[\cc_i,1]| i\in I\}$ is an arbitrary packing of unit balls in $\Ee^3$ and $\mathbf{V}_i$ denotes the Voronoi cell assigned to $\BB^d[\cc_i,1]$, $i\in I$, then 
\[
\frac{\omega_3}{{\rm vol}_3\left(\mathbf{V}_i\cap \BB^3[\cc_i,\sqrt{2}]\right)}\leq \frac{\omega_3}{{\rm vol}_3(\mathbf{D})}< 0.7547<\sigma_3=0.7797...,
\] 
where $\mathbf{D}$ stands for a regular dodacahedron of inradius $1$. Hence, $\hat{\delta}_3< 0.7547$. The latter inequality and (\ref{Main-1}) yield that if $\mathcal {P}$ is an arbitrary LS-packing of $n>1$ unit balls in $\Ee^3$, then 
\[
c(\mathcal {P})\leq \left\lfloor 3n-\hat{\delta}_3^{-\frac{2}{3}}n^{\frac{2}{3}} \right\rfloor\leq\left\lfloor  3n-1.206 n^{\frac{2}{3}}\right\rfloor.
\] 
\end{Remark}

\begin{Remark}\label{lattice-LS-packing}
Let $\mathcal {P}:=\{\BB^d[\cc_i, 1/2]| \cc_i\in \Ze^d, 1\leq i\leq n\}$ be an arbitrary packing of $n$ unit diameter balls with centers having integer coordinates in $\Ee^d$. Clearly, $\mathcal {P}$ is a TS-packing. Then let $c_{ \Ze^d}(n)$ denote the largest $c(\mathcal {P})$ for packings $\mathcal {P}$ of $n$ unit diameter balls of  $\Ee^d$ obtained in this way. It is proved in \cite{BeSzSz} that 
\begin{equation}\label{integer-lattice}
dN^d-dN^{d-1}\leq c_{ \Ze^d}(n)\leq \left\lfloor dn-dn^{\frac{d-1}{d}}\right\rfloor
\end{equation}
for $N\in\Ze$ satisfying $0\leq N\leq n^{1/d}< N+1$, where $d>1$ and $n>1$. Note that if $N=n^{1/d}\in \Ze$, then the lower and upper estimates of (\ref{integer-lattice}) are equal to $c_{ \Ze^d}(n)$. Furthermore,
\begin{equation}\label{integer-lattice-2D} 
c_{ \Ze^2}(n)=\lfloor 2n-2\sqrt{n}\rfloor
\end{equation} 
for all $n>1$. We note that \cite{Ne} (resp., \cite{AlCe}) generates an algorithm that lists some (resp., all) packings  $\mathcal {P}=\{\BB^d[\cc_i, 1/2]| \cc_i\in \Ze^d, 1\leq i\leq n\}$ with $c(\mathcal {P})=c_{ \Ze^d}(n)$ for $d\ge 4$ (resp., $d=2, 3$) and $n>1$. 
\end{Remark}

We cannot resist to close section with the following problem.

\begin{Problem}
 Let $\mathcal {P}$ be an arbitrary LS-packing (resp., TS-packing) of $n>1$ unit diameter balls in $\Ee^d$, $d\geq 3$. Then prove or disprove that $c(\mathcal {P})\leq c_{ \Ze^d}(n)$.
 \end{Problem}

\section*{Acknowledgements}

K\'aroly Bezdek was partially supported by a Natural Sciences and 
Engineering Research Council of Canada Discovery Grant.

Zsolt L\'angi was partially supported by the European Research Council Advanced Grant ``ERMiD'', and the NKFI research grant K147544.

\bibliographystyle{amsalpha}
\bibliography{biblio}
\end{document}